\documentclass[authoryear,preprint,review,12pt]{elsarticle}
\usepackage[utf8]{inputenc}

\usepackage{amssymb}
\usepackage{amsmath}

\journal{Computational Statistics & Data Analysis}
\usepackage[unicode, colorlinks=true, linkcolor=blue, filecolor=cyan, citecolor=blue, urlcolor=magenta, bookmarks=true, bookmarksopen=false, bookmarksnumbered=true,
pdftitle={Robust Low-rank Tensor Regression via Clipping and Huber Loss},
pdfauthor={Li et al.}
]{hyperref}
\usepackage{amsthm}
\usepackage{lmodern}
\usepackage{anyfontsize}
\usepackage{tikz,tikz-3dplot}
\usepackage{stmaryrd}
\usepackage{lipsum}
\usepackage{amsfonts}
\usepackage{graphicx}
\usepackage{epstopdf}
\usepackage{amssymb}
\usepackage{subfigure}
\usepackage{booktabs}
\usepackage{multirow}
\usepackage{diagbox}
\usepackage{pgfplots}
\pgfplotsset{compat=1.18}
\usepackage{algorithm}
\usepackage{algorithmic}
\usepackage{setspace}
\usepackage{arydshln}

\biboptions{semicolon}

\usepackage{enumitem}
\setlist[enumerate]{leftmargin=.5in}
\setlist[itemize]{leftmargin=.5in}

\newtheorem{theorem}{Theorem}
\newtheorem{assumption}{Assumption}
\newtheorem{lemma}{Lemma}

\newtheorem{corollary}{Corollary}
\newtheorem{definition}{Definition}
\newtheorem{remark}{Remark}

\begin{document}

\begin{frontmatter}



\title{Robust Low-rank Tensor Regression via Clipping and Huber Loss}

\author[label1,label2]{Kangqiang Li}
\ead{11935023@zju.edu.cn}
\author[label2]{Bingqi Liu\corref{cor1}}
\ead{bqliu@zju.edu.cn}
\author[label1]{Yang Yang}
\author[label3]{Li Wang}

\affiliation[label1]{organization={Information Center, Hubei Provincial Tobacco Monopoly Administration},
             city={Wuhan},
             citysep={},
             postcode={430030},
             state={Hubei},
             country={China}}
\affiliation[label2]{organization={School of Mathematical Sciences, Zhejiang University},
             city={Hangzhou},
             citysep={},
             postcode={310058},
             state={Zhejiang},
             country={China}}

\affiliation[label3]{organization={Procurement Management Office, Hubei Provincial Tobacco Monopoly Administration},
             city={Wuhan},
             citysep={},
             postcode={430030},
             state={Hubei},
             country={China}}

\cortext[cor1]{Corresponding author}


\begin{abstract}
In this paper, we construct a parameter estimation framework for robust low-rank tensor regression based on a truncation method and Huber loss, specifically focusing on models with random noise having only finite second-order moments. Through a robust gradient descent method, our proposed Huber-type estimator is theoretically optimal in two aspects: (1) its statistical error rate matches the optimal upper bound established for the traditional least squares method under sub-Gaussian error; and (2) the sample complexity for recovering the tensor parameter is also optimal. Extensive numerical experiments demonstrate the robustness of our estimator, indicating that the utilization of truncation and Huber loss significantly enhances stability and statistical effectiveness, outperforming the traditional least squares method. Additionally, the phenomenon of phase transition in the convergence rate of the proposed estimator is confirmed through simulation. Furthermore, applications to image recovery and the Beijing air-quality dataset demonstrate the practical effectiveness of our method.
\end{abstract}




\end{frontmatter}



\section{Introduction}
\label{sec:intro}

In recent years, the tail-robustness issue in parameter regression estimation has become an interesting research highlight. A primary reason for this attention is that many well-used statistical methods and algorithms based on the sub-Gaussian assumption suffer from heavy-tailed distribution and the real-world datasets, especially in the financial area, often exhibit heavy tails. Therefore, some effective approaches to adequately estimate regression parameter with the heavy-tailed noise have been proposed by numerous literature. One of the popular ways is to replace the traditional squared loss with robust alternatives such as absolute loss, quantile loss, Huber loss \citep{Huber1964} and Cauchy loss. This type of robust technique was originally aimed at achieving the outlier-robustness. Recently, \cite{Fan2016} first employed the Huber loss into linear regression problem to investigate the robustness against heavy-tailness of the regression error. Their theoretical result unveils that under only finite second order moment condition on the noise, the proposed robust estimator has the same optimal rate as the case of sub-Gaussian tails via carefully tuning the robustification parameter of Huber loss. \cite{Sun2020} revisited this adaptive Huber regression with only bounded $(1+\epsilon)$-th moment noise, and found a tight phase transition phenomenon for the convergence rate of the regression parameter. \cite{Wang2021} proposed a data-driven approach for Huber-type robust mean estimation and linear regression such that the robustification parameter can efficiently be tuned. There are also a broad range of other literature using the adaptive Huber loss to study regression problems with heavy-tailed errors (see, for example, \cite{Zhou2018} and \cite{Luo2022} for linear regression, \cite{Fan2019} for large-scale multiple testing, \cite{Shen2025} for matrix recovery and \cite{Fan2024} for neural networks).
	
Another convenient and effective way to handle heavy-tailed data is to directly clip the samples, which has become widely used. \cite{Fan2021} proposed a shrinkage principle for low-rank matrix recovery with heavy-tailed data. Specifically, they truncated large heavy-tailed responses or covariates and used the clipped data in the least-squares method. Their theoretical results show that, under mild moment constraints, the robust estimator achieves a similar statistical error rate as in the case with sub-Gaussian tails. Subsequently, \cite{Zhu2021} adopted this robust methodology in generalized linear models and obtained good theoretical and numerical results. Inspired by \cite{Fan2021}, \cite{Wang2023} and \cite{Liu2025} studied high-dimensional vector autoregression with heavy-tailed time series data. Both of their robust estimation procedures required clipping the data vector.

As more and more datasets appear in the form of tensors, conventional methods for studying regression problems based on vector and matrix-valued data are inadequate. Therefore, significant developments have been made in tensor regression. For example, \cite{Lu2020} proposed an estimation procedure to estimate the tensor parameter in high-dimensional quantile regression. \cite{Han2022} introduces a unified statistical and computational framework for generalized low-rank tensor estimation under various probabilistic models, including sub-Gaussian tensor PCA, tensor regression, and Poisson and binomial tensor PCA. Under general deterministic conditions, they establish both statistical error bounds and linear convergence rates. The method is shown to achieve minimax optimal rates in several settings and is validated through simulations and real-data applications such as photon-limited imaging and click-through prediction.

To further fill the gap in robust estimation under tensor regression models and make progress in overcoming heavy-tailed errors from a statistical perspective, this paper investigates tensor regression models with heavy-tailed errors that have only finite second-order moments. By designing a robust gradient descent algorithm, we obtain a Huber-type robust estimator, which theoretically achieves the same optimal estimation error as in the case with sub-Gaussian noise, as shown in \cite{Han2022}. We also optimize the sample complexity through careful proofs under mild constraints. The magnitudes of the robustification parameters used to control the bias and tail robustness are specified with respect to the dimensionality, rank, sample size, and finite second-order moments. Furthermore, to accommodate asymmetric additive errors, we generalize the Huber loss to a broader class of loss functions. Numerical simulations demonstrate that our estimator outperforms that of \cite{Han2022} under both homogeneous and heteroscedastic models. Finally, we apply our method to image recovery. Compared with the original algorithm, the robust Huber loss-based algorithm shows a significant improvement in recovery performance.

The remainder of the paper is organized as follows: Section \ref{SecDef} gives the mathematical notation and definitions used in this paper. Section \ref{sec2} presents our Huber-type parameter estimator for tensor regression with heavy-tailed errors and gives an upper bound in theory. Section \ref{secapp} shows numerical experiments and empirical analysis to illustrate the validity of the methods in this paper. Section \ref{discussion} gives the summary and discussion of this paper. Finally, the proofs of theorems are shown in Appendix.

\section{Notation and definitions}
\label{SecDef}
For any positive integer $n$, we denote the set $\{1,2,\ldots,n\}$ by $[n]$. Uppercase letters denote matrices, while calligraphic letters denote 3-order tensors. For two matrices $X, Y \in \mathbb{R}^{p_{1}\times p_{2}}$ and two tensors $\mathcal{X}, \mathcal{Y} \in \mathbb{R}^{p_{1}\times p_{2}\times p_{3}}$, $\langle X,Y\rangle:=\text{tr}(X^{\top}Y)$ and $\langle \mathcal{X},\mathcal{Y}\rangle:=\sum_{i,j,k}\mathcal{X}_{(i,j,k)}\mathcal{Y}_{(i,j,k)}$.  The Frobenius norms of $X$ and $\mathcal{X}$ are defined as $\|X\|_{F}=\sqrt{\sum_{i,j}X_{(i,j)}^2}$ and $\|\mathcal{X}\|_{F}=\sqrt{\sum_{i,j,k}\mathcal{X}_{(i,j,k)}^2}$, respectively. The nuclear norm and spectral norm are defined as  $\|A\|_{\star}=\text{tr}\left(\sqrt{A^\top A}\right)$ and $\|A\|_{\text{op}}=\sqrt{\lambda_{\max}\left(A^\top A\right)}$. The unit Euclidean sphere on $d$-dimensional space is denoted by $\mathbb{S}^{d-1}$. $\mathbb{O}^{p\times q}:=\left\{\mathbf{U} \in \mathbb{R}^{p \times q}: \mathbf{U}^{\top} \mathbf{U}=\mathbf{I}_q\right\}$.
	As a higher-order generalization of principal component analysis, the Tucker decomposition of a tensor is presented below:
	
\begin{definition}[Tensor Tucker decomposition]\label{def1}
For a given tensor $\mathcal{A}\in \mathbb{R}^{p_{1}\times p_{2}\times p_{3}}$, if there exist a tensor $\mathcal{S}\in \mathbb{R}^{r_{1}\times r_{2}\times r_{3}}$ and a matrix $\mathbf{U}_{i}\in \mathbb{O}^{p_{i}\times r_{i}}$ where $r_{i}\le p_{i}$, $i \in [3]$ such that  $\mathcal{A}=\mathcal{S} \times_{1} \mathbf{U}_{1}\times_{2} \mathbf{U}_{2}\times_{3} \mathbf{U}_{3}$, then $\mathcal{A}$ is said to be the rank-$(r_{1},r_{2},r_{3})$ tensor, and $\mathcal{S}$ is called the kernel tensor. For convenience, denote 
\begin{align*}
\mathcal{A}=\mathcal{S} \times_{1} \mathbf{U}_{1}\times_{2} \mathbf{U}_{2}\times_{3} \mathbf{U}_{3}=:\llbracket \mathcal{S} ; \mathbf{U}_{1}, \mathbf{U}_{2},\mathbf{U}_{3} \rrbracket,	
\end{align*}
where $\left(\mathcal{S} \times_{1} \mathbf{U}_{1}\right)_{i_{1} i_{2} i_{3}}:=\sum_{j=1}^{r_{1}} \mathcal{S}_{j i_{2} i_{3}}\left(\mathbf{U}_{1}\right)_{i_{1} j}$. $\mathcal{S} \times_{2} \mathbf{U}_{2}$ and $\mathcal{S} \times_{3} \mathbf{U}_{3}$ for $\mathbf{U}_{2}\in \mathbb{R}^{p_{2}\times r_{2}},\mathbf{U}_{3}\in \mathbb{R}^{p_{3}\times r_{3}}$ are defined in a similar way.
\end{definition}

Figure \ref{fig1} shows a schematic diagram of the tensor Tucker decomposition.

\newcommand{\Depth}{2.5}
\newcommand{\Height}{2.5}
\newcommand{\Width}{2.5}	
\begin{figure}[htb]
		\centering
		\begin{tikzpicture}
			\coordinate (O) at (-0.5,0,0);
			\coordinate (A) at (-0.5,\Width,0);
			\coordinate (B) at (-0.5,\Width,\Height);
			\coordinate (C) at (-0.5,0,\Height);
			\coordinate (D) at (\Depth-0.2,0,0);
			\coordinate (E) at (\Depth-0.2,\Width,0);
			\coordinate (F) at (\Depth-0.2,\Width,\Height);
			\coordinate (G) at (\Depth-0.2,0,\Height);
			\draw[red!60!black,fill=red!5] (O) -- (C) -- (G) -- (D) -- cycle;
			\draw[red!60!black,fill=red!5] (O) -- (A) -- (E) -- (D) -- cycle;
			\draw[red!60!black,fill=red!5] (O) -- (A) -- (B) -- (C) -- cycle;
			\draw[red!60!black,fill=red!15,opacity=0.8] (D) -- (E) -- (F) -- (G) -- cycle;
			\draw[red!60!black,fill=red!5,opacity=0.6] (C) -- (B) -- (F) -- (G) -- cycle;
			\draw[red!60!black,fill=red!5,opacity=0.8] (A) -- (B) -- (F) -- (E) -- cycle;

			\draw (-0.08,-1.4,0) node {\normalsize{$p_{2}$}};
			\draw (-0.08,-1.1,0) node[rotate = 0] {{\color{gray!75}$\underbrace{\hspace{2.8cm}}$}};
			\draw (-1.1,1.05,2) node[rotate = 0] {\normalsize{$p_{1}$}};
			\draw (-0.8,1.05,2) node[rotate = 270] {{\color{gray!75}$\underbrace{\hspace{2.5cm}}$}};
			\draw (2.8,0,1.55) node[rotate = 45] {\normalsize{$p_{3}$}};
			\draw (2.45,0,1.4) node[rotate = 45] {{\color{gray!75}$\underbrace{\hspace{1.35cm}}$}};
			
			\draw (0.3,0.6,1) node {\Large{$\mathcal{A}$}};
			\draw (3.5,1,1) node {\Large{=}};
			
			\coordinate (O) at (6.3,0.2,0);
			\coordinate (A) at (6.3,\Width/2+0.2,0);
			\coordinate (B) at (6.3,\Width/2+0.2,\Height/2);
			\coordinate (C) at (6.3,0.2,\Height/2);
			\coordinate (D) at (6.6+\Depth/2,0.2,0);
			\coordinate (E) at (6.6+\Depth/2,\Width/2+0.2,0);
			\coordinate (F) at (6.6+\Depth/2,\Width/2+0.2,\Height/2);
			\coordinate (G) at (6.6+\Depth/2,0.2,\Height/2);
			
			\coordinate (B1) at (4.3,\Width/2+0.7,\Height/2);
			\coordinate (C1) at (4.3,-1+0.7,\Height/2);
			\coordinate (F1) at (4.3+\Depth/2,\Width/2+0.7,\Height/2);
			\coordinate (G1) at (4.3+\Depth/2,-1+0.7,\Height/2);
			\draw[black!60!black,fill=black!5] (O) -- (C) -- (G) -- (D) -- cycle;
			\draw[black!60!black,fill=black!5] (O) -- (A) -- (E) -- (D) -- cycle;
			\draw[black!60!black,fill=black!5] (O) -- (A) -- (B) -- (C) -- cycle;
			\draw[black!60!black,fill=black!15,opacity=0.8] (D) -- (E) -- (F) -- (G) -- cycle;
			\draw[black!60!black,fill=black!5,opacity=0.6] (C) -- (B) -- (F) -- (G) -- cycle;
			\draw[black!60!black,fill=black!5,opacity=0.8] (A) -- (B) -- (F) -- (E) -- cycle;
			\draw (7.1,-0.1,\Height/2) node {\normalsize{$r_{1}\times r_{2}\times r_{3}$}};
			\draw (7.1,0.85,\Height/2) node {\large{$\mathcal{S}$}};
			
			\draw (4.9,-0.7,1) node {\normalsize{$p_{1}\times r_{1}$}};
			
			\draw[blue!60!black,fill=blue!5] (C1) -- (B1) -- (F1) -- (G1) -- cycle;
			
			\coordinate (A2) at (7.3,\Width/2+1.2,-1);
			\coordinate (B2) at (7.3,\Width/2+1.2,\Height-1);
			\coordinate (E2) at (7.5+\Depth/2,\Width/2+1.2,-1);
			\coordinate (F2) at (7.5+\Depth/2,\Width/2+1.2,\Height-1);
			\draw[orange!60!black,fill=orange!5,opacity=1] (A2) -- (B2) -- (F2) -- (E2) -- cycle;
			
			\coordinate (O3) at (8.35,-1,0);
			\coordinate (A3) at (8.35,\Width/2-0.2,0);
			\coordinate (D3) at (8.35+\Depth/2,-1,0);
			\coordinate (E3) at (8.35+\Depth/2,\Width/2-0.2,0);
			\draw[green!60!black,fill=green!5] (O3) -- (A3) -- (E3) -- (D3) -- cycle;
			
			\draw (5.8,0.7,1) node {\small{$\times_{1}$}};
			\draw (8.35,0.3,1) node {\small{$\times_{3}$}};
			\draw (7.1,1.43,-0.65) node {\small{$\times_{2}$}};
			
			\draw (9,-1.3,0) node {\normalsize{$p_{3}\times r_{3}$}};
			\draw (8.95,0,0) node {\normalsize{$\mathbf{U_{3}}$}};
			\draw (4.8,0.7,1) node {\normalsize{$\mathbf{U_{1}}$}};
			\draw (7,1.4,-2.3) node {\normalsize{$\mathbf{U_{2}}$}};
			\draw (5.6,1.5,-2.2) node {\normalsize{$p_{2}\times r_{2}$}};
		\end{tikzpicture}
		\caption{Tucker decomposition of  tensor $\mathcal{A}=\llbracket \mathcal{S} ; \mathbf{U}_{1}, \mathbf{U}_{2},\mathbf{U}_{3} \rrbracket\in \mathbb{R}^{p_{1}\times p_{2}\times p_{3}}$.}
		\label{fig1}
\end{figure}
The matricization operator $\mathcal{M}_{k}: \mathbb{R}^{p_{1}\times p_{2}\times p_{3}}\to \mathbb{R}^{p_{k}\times (p_{k+1}p_{k+2})}$ is defined as
\begin{align*}
	\left[\mathcal{M}_{k}(\mathcal{A})\right]_{i_{k}, i_{k+1}+p_{k+1}\left(i_{k+2}-1\right)}=\mathcal{A}_{i_{1} i_{2} i_{3}}\quad \text{for all}\quad \mathcal{A}\in \mathbb{R}^{p_{1}\times p_{2}\times p_{3}},
\end{align*}
where $k+1$ and $k+2$ are computed with modulo $3$. The low-rank tensor can be Tucker decomposed by the following higher-order singular value decomposition algorithm proposed by \cite{DeLathauwer2000}.

\begin{algorithm}
    \caption{High order singular value decomposition (HOSVD)}
    \label{Al0}
    \begin{algorithmic}[1]
        \REQUIRE rank-$\left(r_1, r_2, r_3\right)$ tensor $\mathcal{Y} \in \mathbb{R}^{p_1 \times p_2 \times p_3}$.
        \STATE $\mathbf{U}_k=\operatorname{SVD}_{r_k}\left(\mathcal{M}_k(\mathcal{Y})\right), k=1,2,3$
        \STATE $\mathcal{S}=\llbracket \mathcal{Y} ; \mathbf{U}_1^{\top}, \mathbf{U}_2^{\top}, \mathbf{U}_3^{\top} \rrbracket$
        \ENSURE $\left(\mathcal{S}, \mathbf{U}_1, \mathbf{U}_2, \mathbf{U}_3\right)$.
    \end{algorithmic}
\end{algorithm}

Let $\kappa=\bar{\lambda}/\underline{\lambda}$. Here, we define $\bar{\lambda}:=\max \left\{\left\|\mathcal{M}_{k}\left(\mathcal{A}^{*}\right)\right\|_{\text{op}}: k\in[3]\right\}$ and $\underline{\lambda}:=\min\left\{\sigma_{r_{k}}\left(\mathcal{M}_{k}\left(\mathcal{A}^{*}\right)\right): k\in[3]\right\}$. Furthermore, let $\bar{p}=\max\{p_{k}:k\in [3]\}$ and $\bar{r} = \max\{r_{k}: k\in [3]\}$. For a differentiable function $f: \mathbb{R}^{p_{1}\times p_{2}\times p_{3}}\to \mathbb{R}$, we write $\nabla f$ as its gradient function. Given two sequences $\{a_{n}\}_{n=1}^{\infty}$ and $\{b_{n}\}_{n=1}^{\infty}$, we use the notation $a_{n} \asymp b_{n}$ if there exist positive constants $C_1$ and $C_2$ such that $C_1 b_{n} \le a_{n} \le C_2 b_{n}$ for all $n$.

\begin{definition}[Huber loss \citep{Huber1964}]\label{def2}
The Huber loss $\ell_{\varpi}(\cdot)$ is defined as
\begin{align*}
	\ell_{\varpi}(x)=\begin{cases}x^{2}/2,&\text{if } |x|\leq\varpi;\\ \varpi|x|-\varpi^{2}/2,&\text{if  } |x|>\varpi,\end{cases}
\end{align*}
where $\varpi$ is a robustification parameter which strikes a balance between the induced bias and tail robustness.
\end{definition}

From the above definition, we obtain the following truncation function by deriving the Huber loss:
\begin{definition}[Clipping function]\label{def3}
Let $\psi_{\tau}(\cdot)$ be a clipping function defined by 
\begin{align*}
	\psi_{\tau}(x)=(|x|\wedge\tau) \operatorname{sign} (x),\quad x\in \mathbb{R},
\end{align*}
where $\tau$ is a robustification parameter. 
\end{definition}

Figure \ref{fig2} shows the image of Huber loss, squared loss and truncation function.

\begin{figure}[htb]
\centering
\subfigure{
			\begin{tikzpicture}[scale=0.5]
				\begin{axis}[ymax=5,ymin=-2,xlabel = $x$,ylabel = {$f(x)$}, line width=0.7mm,legend columns=1,legend pos=south east,title style={font=\Large},legend style={font=\Large},font=\Large,height=13cm, width=15cm, xtick={-3,-2,-1,0,1,2,3},  ymajorgrids=true, xmajorgrids=true]
					\addplot [
					domain=-3:3,
					samples=100,
					color=red,]
					{x^2/2};
					\addplot [
					domain=-1:1,
					samples=100,
					color=blue,
					]
					{x^2/2};
					\addplot [
					domain=1:3,
					samples=100,
					color=blue,
					]
					{x-1/2};
					\addplot [
					domain=-3:-1,
					samples=100,
					color=blue,
					]
					{-1/2-x};
					\addplot [
					domain=-1:1,
					samples=100,
					color=green,
					]
					{x};
					\addplot [
					domain=1:3,
					samples=100,
					color=green,
					]
					{1};
					\addplot [
					domain=-3:-1,
					samples=100,
					color=green,
					]
					{-1};
					\legend{$x^2/2$, $\ell_{\varpi}(x):\varpi=1$,,,$\psi_{\tau}(x):\tau=1$}
				\end{axis}
		\end{tikzpicture}}
\caption{Illustration of Huber loss, square loss and truncation function.}
\label{fig2}
\end{figure}

\section{Main result}\label{sec2}
\subsection{Model and theoretical results}
Suppose $n$ i.i.d. samples $\{y_{i},\mathcal{X}_{i}\}_{i=1}^{n}$ are generated from the following tensor regression:
\begin{equation}\label{eq}
y_{i}=\left\langle\mathcal{X}_{i}, \mathcal{A}^{*}\right\rangle+\varepsilon_{i}\quad\text{with}\quad\mathbb{E}\left(\varepsilon_{i} \mid \mathcal{X}_{i}\right)=0 \quad\text{and}\quad M=\sqrt[k]{\mathbb{E}\left(\mathbb{E}\left(\varepsilon_{i}^{2}|\mathcal{X}_{i}\right)^k\right)}<\infty,
\end{equation}
where $\mathcal{A}^{*} \in \mathbb{R}^{p_{1} \times p_{2} \times p_{3}}$ is a rank-$(r_1, r_2, r_3)$ tensor with $r_{k} \ll p_{k}$, to be estimated and $\{\mathcal{X}_{i}\}_{i=1}^{n}$ are tensor covariates. Since the noise only possesses mild polynomial moment, the traditional least square method is not efficient in estimating the tensor parameter $\mathcal{A}^{*}$ and sensitive to outliers and heavy-tailed data. In order to handle heavy-tailed noise, we propose to utilize Huber loss in Definition \ref{def2} and define the empirical loss function $\mathcal{L}_{\varpi}(\mathcal{A})=\frac{1}{n}\sum_{i=1}^{n}\ell_{\varpi}(y_{i}-\left\langle\mathcal{X}_{i}, \mathcal{A}\right\rangle)$. We solve the following optimization problem to obtain a low-rank estimator for $\mathcal{A}^{*}$ under the heavy-tailed setting:
\begin{align*}
\left(\widehat{\mathcal{S}}, \widehat{\mathbf{U}}_{1}, \widehat{\mathbf{U}}_{2}, \widehat{\mathbf{U}}_{3}\right)=\underset{\mathcal{S}, \mathbf{U}_{1}, \mathbf{U}_{2},\mathbf{U}_{3}}{\arg \min }\left\{\mathcal{L}_{\varpi}\left(\llbracket \mathcal{S} ; \mathbf{U}_{1}, \mathbf{U}_{2},\mathbf{U}_{3} \rrbracket\right)+\frac{a}{2} \sum_{k=1}^{3}\left\|\mathbf{U}_{k}^{\top} \mathbf{U}_{k}-b^{2} \mathbf{I}_{r_{k}}\right\|_{\mathrm{F}}^{2}\right\},	
\end{align*}
where $a,b>0$ are tuning parameters. The regular term $\frac{a}{2} \sum_{k=1}^{3}\left\|\mathbf{U}_{k}^{\top} \mathbf{U}_{k}-b^{2} \mathbf{I}_{r_{k}}\right\|_{\mathrm{F}}^{2}$ ensures that in the process of using gradient descent method, $\{\mathbf{U}_{k}\}_{k=1}^3$ are nonsingular without changing the minimum point. Before we employ the gradient descent method to get an excellent approximation of $\left\llbracket\widehat{\mathcal{S}}; \widehat{\mathbf{U}}_{1}, \widehat{\mathbf{U}}_{2}, \widehat{\mathbf{U}}_{3}\right\rrbracket$, a reliable initialization is necessary. \cite{Han2022} constructed an unbiased estimator $\frac{1}{n}\sum_{i=1}^ny_{i}\mathcal{X}_{i}$ based on the normality assumption for noise, but our noise distribution possesses a heavy-tail. Therefore, the response variable $y_{i}$ needs to be trimmed via the clipping function defined in Definition \ref{def3}.

By the above method, we use the following robust gradient descent algorithm to obtain our Huber-type robust estimator.

\begin{algorithm}[htb]
    \caption{The Robust Gradient Descent Algorithm}
    \label{Al1}
    \begin{algorithmic}[1]
        \REQUIRE $\{y_{i},\mathcal{X}_{i}\}_{i=1}^{n}$, robustification parameters $\tau$ and $\varpi$, tuning parameters $a$ and $b$, step size $\eta$, the number of iteration $T_{\max}$.
        \STATE $\widetilde{\mathcal{A}}=\frac{1}{n}\sum_{i=1}^{n}\psi_{\tau}(y_{i})\mathcal{X}_{i}$
        \STATE $\left(\widetilde{\mathcal{S}},\widetilde{\mathbf{U}}_{1},\widetilde{\mathbf{U}}_{2}, \widetilde{\mathbf{U}}_{3}\right)=\text{HOSVD}\left(\widetilde{\mathcal{A}}\right)$ (HOSVD in Algorithm \ref{Al0})
        \STATE Let $\mathbf{U}_{k}^{(0)}=b\widetilde{\mathbf{U}}_{k}$ for $k\in [3]$ and $\mathcal{S}^{(0)}=\widetilde{\mathcal{S}}/b^3$.
        \FOR{$T=0,1,2,\ldots,T_{\max}-1$}
            \STATE For each $k\in [3]$, we compute the following equation:
            \STATE $\mathbf{U}_{k}^{(T+1)}=\mathbf{U}_{k}^{(T)}-\eta\left(\nabla_{\mathbf{U}_{k}} \mathcal{L}_{\varpi}\left(\mathcal{S}^{(T)}, \dots, \mathbf{U}_{3}^{(T)}\right)+a \mathbf{U}_{k}^{(T)}\left(\mathbf{U}_{k}^{(T) \top} \mathbf{U}_{k}^{(T)}-b^{2} \mathbf{I}\right)\right).$
            \STATE Then:
            \STATE $\mathcal{S}^{(T+1)}=\mathcal{S}^{(T)}-\eta \nabla_{\mathcal{S}} \mathcal{L}_{\varpi}\left(\mathcal{S}^{(T)}, \mathbf{U}_{1}^{(T)}, \mathbf{U}_{2}^{(T)}, \mathbf{U}_{3}^{(T)}\right)$.
        \ENDFOR
        \ENSURE $\mathcal{A}^{(T_{\max})}=\mathcal{S}^{(T_{\max})}\times_{1}\mathbf{U}_{1}^{(T_{\max})} \times_{2}\mathbf{U}_{2}^{(T_{\max})}\times_{3}\mathbf{U}_{3}^{(T_{\max})}$.
    \end{algorithmic}
\end{algorithm}

There are two main differences in Algorithm \ref{Al1} compared with \cite{Han2022} such that it possesses the robustness against heavy-tailed errors of noise:

\begin{enumerate}[label=(\arabic*)]
	\item The initializer $\widetilde{\mathcal{A}}$ contains a robustification parameter $\tau$ which clips the response $\{y_{i}\}_{i=1}^n$. the clipping technique has been successfully applied by \cite{Fan2021} and \cite{Zhu2021}. Some specific benefits of doing so are stated in Remark \ref{remark1}.
		\item We adopt Huber loss rather than the ordinary least squares considered by \cite{Han2022}.
\end{enumerate}

The statistical guarantee of $\mathcal{A}^{(T_{\max})}$ is presented as follows.
\begin{theorem}\label{theorem1}
	Assume that the following conditions hold:
	\begin{enumerate}[label=(\arabic*)]
		\item Let all entries of $\mathcal{X}$ be i.i.d. sampled from sub-Gaussian distribution with mean-zero, variance-one and $\left\|\mathcal{X}_{(j,k,l)}\right\|_{\psi_{2}}\leq K<\infty$. Assume that  $\{\mathcal{X}_{i}\}_{i=1}^{n}$ are i. i. d. copies of $\mathcal{X}$ and $\{y_{i}\}_{i=1}^n$ are i.i.d. from (\ref{eq}).
		\item Denote $\bar{\lambda}:=\max \left\{\left\|\mathcal{M}_{k}\left(\mathcal{A}^{*}\right)\right\|_{\text{op}}: k\in[3]\right\}$ and $\underline{\lambda}:=\min\left\{\sigma_{r_{k}}\left(\mathcal{M}_{k}
		\left(\mathcal{A}^{*}\right)\right): k\in[3]\right\}$, where $\bar{p}=\max\{p_{k}:k\in [3]\}$ and $\bar{r}=\max\{r_{k}:k\in [3]\}$. There exist some positive constants $\{c_{i}\}_{i=1}^3$ such that $\underline{r}\geq c_{1}\bar{r}$, $M\leq c_{2} \|\mathcal{A}^{*}\|_{F}^2$, $\underline{r}\geq \left(\sqrt{\bar{p}}\log(\bar{p})\right)^{\frac{1}{3}}$ and $\underline{\lambda}\geq c_{3}$.
	\end{enumerate}

If $\varpi\asymp_{K,k}\left(Mn/df\right)^{\frac{1}{2}}, \tau\asymp_{K,k}\left(Rn/df\right)^{\frac{1}{2}}$, $b \asymp \bar{\lambda}^{1 / 4}$ and $a\asymp\frac{\bar{\lambda}}{\kappa^{2}}$ are chosen, then for $\forall t> \log(13)$, there exist positive constants $c_{0},c_{1},c_{2}$ and $\{C_{i}\}_{i=1}^5$ such that as long as $\eta>c_{0}\bar{\lambda}^{-\frac{3}{2}}$,
\begin{align*}
	n>C_{1} \max\left\{{R}^2\bar{p}^{\frac{3}{2}}\bar{r}, R\kappa^2\bar{p}^{\frac{3}{2}}\sqrt{\bar{r}},\kappa^6\bar{p}\bar{r}^2\right\}\quad\text{and}\quad T_{\max}>C_{2} \log\left(\frac{n\underline{\lambda}}{M df\kappa^{4}}\right),
\end{align*}
with probability at least $1-6\exp(-c_{1}\bar{p})-4\exp\left(df(\log(13)-t)\right)\allowbreak -C_{3}T_{\max}\allowbreak\exp(-C_{4}df)\allowbreak -(T_{\max}+4)e^{-\bar{p}^{\frac{3}{2}}\bar{r}}-4\bar{p}^{-c_{2}}$, we have
\begin{align*}
	\left\|\mathcal{A}^{(T_{\max})}-\mathcal{A}^{*}\right\|_{F}<C_{5} \kappa^{2}\sqrt{Mdf/n}\left(\sqrt{t}+t\right),
\end{align*}
where $df:=r_{1}r_{2}r_{3}+\sum_{i=1}^3p_{i}r_{i}$, $\kappa=\bar{\lambda}/\underline{\lambda}$ and $R:=\mathbb{E}[y_{i}^2]$.
\end{theorem}

\begin{remark}\label{remark1}
From Theorem \ref{theorem1}, it follows that under the bounded second-order moments condition, our proposed estimator realizes the minmax optimal rate of convergence established by \cite{Han2022}. On the other hand, the sample complexity $n\gtrsim \bar{p}^{\frac{3}{2}}\bar{r}$ is also optimal if $R$ is smaller than some fixed constant.
\end{remark}

\begin{remark}\label{remark2}
In the presence of heavy-tailed noises, the initializer $\widetilde{\mathcal{A}}$ offers some key virtues in both theoretical and numerical aspects. First, after appropriate truncation, $\widetilde{\mathcal{A}}$ has a non-asymptotic upper bound with an exponential-type exception probability. Therefore, it is a good initialization and well reflects the parameter tensor $\mathcal{A}^{*}$, which can provide a firm foundation for rank-$(r_{1}, r_{2}, r_{3})$ estimation. Besides, $\widetilde{\mathcal{A}}$ can reduce the required number of iterations $T_{\max}$ and enhance the stability of Algorithm \ref{Al1}.
\end{remark}

\begin{remark}\label{remark3}
The constraints on $\underline{r}$ in condition (2) is easily satisfied. For example, if $\bar{p}=10^4$, we only need $\underline{r}\geq 10$. Note that the scales of $\tau$ and $\varpi$ are related to the rank-$(r_1, r_2,r_3)$, but in practice, the rank is often unknown. Therefore, in order to overcome this drawback, the estimate of  $\{r_{k}\}_{k=1}^{3}$ can be obtained by substituting the robust initial value $\widetilde{\mathcal{A}}$ into the rank selection method of \cite{Han2022}:
\begin{align*}
\hat{r}_k=\max \left\{r: \sigma_r\left(\mathcal{M}_k(\widetilde{\mathcal{A}})\right) \geq c\delta_k\right\}, \quad k=1,2,3,	
\end{align*}
where $\delta_k:=\operatorname{Median}\left\{\sigma_1\left(\mathcal{M}_k(\widetilde{\mathcal{A}})\right), \ldots, \sigma_{p_k}\left(\mathcal{M}_k(\widetilde{\mathcal{A}})\right)\right\}$  and $c>0$ are threshold levels.
\end{remark}

\begin{remark}\label{remark4}
Without considering the theoretical upper bound on the initial value $\widetilde{\mathcal{A}}$, from the proof of Theorem \ref{theorem1}, we can further obtain that when the error term has only finite $(1+\delta)$-order moments, i.e., for $\delta>0$, $M_{\delta}:=\sqrt[k]{\mathbb{E}\left(\mathbb{E}\left(|\varepsilon_{i}|^{1+\delta}|\mathcal{X}_{i}\right)^k\right)}<\infty$, if we choose $\varpi \asymp_{K, k}\left(M_\delta^{1 / k} n / d f\right)^{\frac{1}{1+\delta}}$, then we have
\begin{align*}
\left\|\mathcal{A}^{\left(T_{\max }\right)}-\mathcal{A}^{*}\right\|_F\lesssim \kappa^2 M_\delta^{\frac{1}{k(1+\delta)}}\left(\frac{d f}{n}\right)^{\frac{\delta}{1+\delta}} \text{with high probability holds.}	
\end{align*}
It can be seen that the estimator has a smooth phase transition when $0<\delta<1$, which is similar to the adaptive Huber linear regression established by \cite{Sun2020}. The simulation experiments in next section confirm the phenomenon.
\end{remark}

The following corollary clarifies that as the sample size increases, we do not need the additional constraint to be added to the rank $\underline{r}$.

\begin{corollary}
Without the constraint condition $\underline{r}\geq \left(\sqrt{\bar{p}}\log(\bar{p})\right)^{\frac{1}{3}}$, the conclusion of Theorem \ref{theorem1} would still hold as long as 
\begin{align*}
	\tau\asymp_{K,k}\|\mathcal{A}^{*}\|_{F}\left(n/df\right)^{\frac{1}{2}}, \quad n\gtrsim \allowbreak \max\left\{\kappa^4 \bar{p}^2 \log (\bar{p}),\allowbreak \kappa^4 \bar{p}^{\frac{3}{2}} \bar{r}^3, \kappa^6\bar{p}\bar{r}^2\right\}
\end{align*}
are chosen.
\end{corollary}

\subsection{Generalization of the loss function}
	Since different robust loss functions have different results when dealing with different types of heavy-tailed errors, in this subsection we generalize Huber loss to asymmetric loss functions \citep{Man2024} which only need to satisfy the following assumptions: 
	\begin{assumption}\label{assumption}
		Let $\ell_\varpi(x)$ fulfill the following conditions:
		\begin{enumerate}[label=(\arabic*)]
			\item $\ell_\varpi(x)=\varpi^2 \ell_{1}(x / \varpi)$, where $\ell_{1}: \mathbb{R} \mapsto[0, \infty)$;
			\item $\ell_{1}^{\prime}(0)=0$, for $\forall x \in \mathbb{R}$, $\left|\ell_{1}^{\prime}(x)\right| \leq \min \left(c_1,|x|\right)$ and $\left|\ell_{1}^{\prime}(x)-x\right| \leq c_{2}x^2$;
			\item $\ell_{1}^{\prime \prime}(0)=1$ and for $\forall |x| \leq c_3$, $\ell_{1}^{\prime \prime}(x) \geq c_4$, where $\{c_{i}\}_{i=1}^4$ is a positive constant.
	\end{enumerate}
\end{assumption}

Assumption \ref{assumption}  contains a variety of loss functions such as Huber loss, Tukey's biweight loss
\begin{align*}
	\ell_{\varpi}(x)= \begin{cases}\frac{\varpi^2}{6}\left(1-\left(1-\frac{x^2}{\varpi^2}\right)^3\right), & \text { if }|x| \leq \varpi, \\ \frac{\varpi^2}{6}, & \text { if }|x|>\varpi.\end{cases}
\end{align*}
and Cauchy loss $\ell_{\varpi}(x)=\frac{\varpi^2}{2} \log \left(1+\frac{x^2}{\varpi^2}\right)$. Using a robust approach similar to the previous section, define the empirical loss as $\mathcal{L}_{\varpi}(\mathcal{A})=\frac{1}{n}\sum_{i=1}^{n}\ell_{\varpi}(y_{i}-\left\langle\mathcal{X}_{i}, \mathcal{A}\right\rangle)$ and for the estimation of the initial values, we use the derivative function satisfying Assumption \ref{assumption} to truncate the response variable and obtain $\widetilde{\mathcal{A}}=\frac{1}{n}\sum_{i=1}^{n}\ell_{\tau}^{\prime}(y_{i})\mathcal{X}_{i}$. Based on this, the following theorem shows that the estimates produced by Algorithm \ref{Al1} are still optimal in terms of the statistical error and sample complexity.
\begin{theorem}\label{theorem2}
	Under Assumption \ref{assumption} and the conditions of Theorem \ref{theorem1}, the following conclusion holds: 
	For $\forall t> \log(13)$, with probability at least $1-6\exp(-c\bar{p})-4\exp\left(df(\log(13)-t)\right)-C_{2}T_{\max}\exp(-C_{3}df)-(T_{\max}+4)e^{-\bar{p}^{\frac{3}{2}}\bar{r}}-4\bar{p}^{-c}$,
	we have
	\begin{align*}
		\left\|\mathcal{A}^{(T_{\max})}-\mathcal{A}^{*}\right\|_{F}<C_{4} \kappa^{2}\sqrt{Mdf/n}\left(\sqrt{t}+t\right).
	\end{align*}
\end{theorem}

\section{Simulation analysis and application}\label{secapp}
\subsection{Numerical experiments}
In this section, we illustrate the statistical performance of Algorithm \ref{Al1}. Let $r_{1}=r_{2}=r_{3}=r$ and $p_{1}=p_{2}=p_{3}=p$ to facilitate simulation. The parameter tensor $\mathcal{A}^{*}=\llbracket \mathcal{S}^{\star} ; \mathbf{U}_{1}^{\star}, \mathbf{U}_{2}^{\star},\mathbf{U}_{3}^{\star} \rrbracket$ is constructed by the following two steps:
\begin{enumerate}[label=(\arabic*)]
	\item $\mathcal{S}^{\star}=\lambda\mathcal{S}/\min\left\{\sigma_{r}\left(\mathcal{M}_{k}
	\left(\mathcal{S}\right)\right): k\in[3]\right\}$ where each entry of $\mathcal{S}$ is drawn from $\mathcal{N}(0,1)$ and $\lambda$ is a specified constant.
	\item For each $k\in [3]$, $\mathbf{U}_{k}^{\star}$ is top $r$ eigenvectors of $p$-dimensional sample covariance matrix from 100 i.i.d. standard Gaussian random vectors. Tensor covariates $\{\mathcal{X}_{i}\}_{i=1}^{n}$ have i.i.d. entries which are drawn from $\mathcal{N}(0,1)$.
\end{enumerate}

We compare Algorithm \ref{Al1} with the algorithm of \cite{Han2022} and abbreviate them as RGD and GD, respectively. the results are shown in Figure \ref{fig3} and Table \ref{tab1}. The experimental results are based on 200 independent repetitions. From Figure \ref{fig3} with error bars, two algorithms possess similar statistical effects after 300 iterations, when the random noise follows standard normal distribution. However, when the noise obeys a scaled Student's $t$-distribution with $2.1$ degrees of freedom, RGD performs significantly better than GD. Table \ref{tab1} presents parameter values in the experiments and shows that the initial value  $\widetilde{\mathcal{A}}$ via using the truncation method outperforms the original estimator in terms of mean and standard deviation. There is a certain probability that the algorithm GD will fail to converge in 200 independently repeated experiments. The results in Figure \ref{fig3} are calculated after eliminating these non-converging data, which show that our algorithm is more robust.

\begin{figure}[htbp]
    \centering
    \begin{minipage}[t]{0.48\linewidth}
        \centering
        \begin{tikzpicture}
            \begin{axis}[
                width=\linewidth, 
                title={$(a): r=3$, $p=20$, $n=1000$, $\mathcal{N}(0,1)$ noise},
                xlabel = $T$,
                ylabel = {$\text{error}$}, 
                title style={font=\small},
                legend pos=north east, 
                legend columns=1,
                font=\tiny,
                height=6cm, 
                xtick={1,100,200,300,400,500},
                ytick={0,2.619329,5,10,15},
                ymajorgrids=true
            ]
                \addplot[color=blue,mark=+, error bars/.cd, y dir=both, y explicit] 
                coordinates {
                    (1,9.792853)+-(0,1.850605)(11,7.682936)+-(0,1.338879)(21,6.770298)+-(0,1.428238) (31,6.188168)+-(0,1.542553) (41,5.762978)+-(0,1.643557) (51,5.430366)+-(0,1.727951)
                    (61,5.159971)+-(0,1.796524)(71,4.934913)+-(0,1.850961)  (81,4.744587)+-(0,1.893123)  (91,4.581659)+-(0,1.924731) (101,4.440683)+-(0,1.947281)  (111,4.317421)+-(0,1.962063)  (121,4.208507)+-(0,1.970216)  (131,4.111248)+-(0,1.972789)
                    (141,4.023509)+-(0,1.970803)  (151,3.943634)+-(0,1.965296)  (161,3.870371)+-(0,1.957312)  (171,3.802790)+-(0,1.947797) (181,3.740172)+-(0,1.937459)  (191,3.681905)+-(0,1.926705)  (201,3.627426)+-(0,1.915710)(211,3.576212)+-(0,1.904527) (221,3.527797)+-(0,1.893175)  (231,3.481777)+-(0,1.881676)  (241,3.437818)+-(0,1.870060)  (251,3.395644)+-(0,1.858362)  (261,3.355027)+-(0,1.846615)  (271,3.315781)+-(0,1.834849) (281,3.277753)+-(0,1.823083)  (291,3.240814)+-(0,1.811333)  (301,3.204858)+-(0,1.799606)  (311,3.169795)+-(0,1.787906)  (321,3.135551)+-(0,1.776232)  (331,3.102062)+-(0,1.764582)  (341,3.069275)+-(0,1.752951)(351,3.037145)+-(0,1.741335)  (361,3.005635)+-(0,1.729729)  (371,2.974712)+-(0,1.718129)  (381,2.944350)+-(0,1.706533)  (391,2.914526)+-(0,1.694940)(401,2.885218)+-(0,1.683351)
                    (411,2.856408)+-(0,1.671767)  (421,2.828081)+-(0,1.660190)  (431,2.800218)+-(0,1.648621)  (441,2.772806)+-(0,1.637061)(451,2.745827)+-(0,1.625507)  (461,2.719264)+-(0,1.613957)  (471,2.693100)+-(0,1.602405)  (481,2.667319)+-(0,1.590847)  (491,2.641904)+-(0,1.579278)(500,2.619329)+-(0,1.568855) 
                };
                \addplot[color=red,mark=x,error bars/.cd,y dir=both, y explicit] 
                coordinates {
                    (1,10.705529)+-(0,3.512661)(11,9.142059)+-(0,2.617296)(21,8.159941)+-(0,2.364309) (31,7.452470)+-(0,2.269923) (41,6.914680)+-(0,2.250241) (51,6.490852)+-(0,2.268852)
                    (61,6.146242)+-(0,2.304229)(71,5.858635)+-(0,2.342916)  (81,5.613665)+-(0,2.377725)  (91,5.401523)+-(0,2.405485) (101,5.215460)+-(0,2.425458)  (111,5.050579)+-(0,2.438023)  (121,4.903109)+-(0,2.443932)  (131,4.770059)+-(0,2.444076)
                    (141,4.649012)+-(0,2.439337)  (151,4.537945)+-(0,2.430481)  (161,4.435225)+-(0,2.418182)  (171,4.339537)+-(0,2.403067) (181,4.249843)+-(0,2.385681)  (191,4.165314)+-(0,2.366517)  (201,4.085264)+-(0,2.345988)(211,4.009136)+-(0,2.324433) (221,3.936451)+-(0,2.302123)  (231,3.866825)+-(0,2.279317)  (241,3.799951)+-(0,2.256272)  (251,3.735598)+-(0,2.233254)  (261,3.673579)+-(0,2.210495)  (271,3.613732)+-(0,2.188154) (281,3.555908)+-(0,2.166333)  (291,3.499994)+-(0,2.145114)  (301,3.445869)+-(0,2.124522)  (311,3.393424)+-(0,2.104564)  (321,3.342560)+-(0,2.085220)  (331,3.293187)+-(0,2.066444)  (341,3.245226)+-(0,2.048190)(351,3.198608)+-(0,2.030418)  (361,3.153249)+-(0,2.013054)  (371,3.109076)+-(0,1.996035)  (381,3.066031)+-(0,1.979324)  (391,3.024054)+-(0,1.962875)(401,2.983082)+-(0,1.946628)
                    (411,2.943060)+-(0,1.930557)  (421,2.903950)+-(0,1.914655)  (431,2.865708)+-(0,1.898911)  (441,2.828292)+-(0,1.883306)(451,2.791673)+-(0,1.867855)  (461,2.755824)+-(0,1.852575)  (471,2.720713)+-(0,1.837468)  (481,2.686321)+-(0,1.822549)  (491,2.652646)+-(0,1.807866)(500,2.622941)+-(0,1.794864) 
                };
                \legend{GD ,RGD}
            \end{axis}
        \end{tikzpicture}
    \end{minipage}%
    \hfill
    \begin{minipage}[t]{0.48\linewidth}
        \centering
        \begin{tikzpicture}
            \begin{axis}[
                width=\linewidth,
                title={$(b): r=3$, $p=15$, $n=1000$, $10t_{2.1}$ noise},
                xlabel = $T$,
                ylabel = {$\text{error}$}, 
                title style={font=\small},
                legend pos=north west, 
                legend columns=2,
                font=\tiny,
                height=6cm, 
                xtick={1,100,200,300,400},
                ytick={5.54,6.68,7.82,10.78,14.83381,18.89},
                ymajorgrids=true
            ]
                \addplot[color=blue,mark=+,densely dotted,error bars/.cd, y dir=both, y explicit] 
                coordinates { 
                    (1,13.96526)+-(0,1.722473)(11,13.92646)+-(0,2.723122)(21,14.14666)+-(0,3.052672) (31,14.29972)+-(0,3.232899) (41,14.40810)+-(0,3.352288) (51,14.48597)+-(0,3.435386)
                    (61,14.54212)+-(0,3.493755)(71,14.58323)+-(0,3.536220)  (81,14.61435)+-(0,3.569478)  (91,14.63853)+-(0,3.597299) (101,14.65772)+-(0,3.621515)  (111,14.67332)+-(0,3.643259)  (121,14.68638)+-(0,3.663597)  (131,14.69772)+-(0,3.683794)
                    (141,14.70816)+-(0,3.705068)  (151,14.71826)+-(0,3.727490)  (161,14.72792)+-(0,3.749580)  (171,14.73691)+-(0,3.770190) (181,14.74518)+-(0,3.789282)  (191,14.75273)+-(0,3.807143)  (201,14.75958)+-(0,3.823999)(211,14.76573)+-(0,3.840017) (221,14.77122)+-(0,3.855353)  (231,14.77607)+-(0,3.870154)  (241,14.78035)+-(0,3.884542)  (251,14.78415)+-(0,3.898599)  (261,14.78755)+-(0,3.912356)  (271,14.79064)+-(0,3.925796) (281,14.79352)+-(0,3.938854)  (291,14.79623)+-(0,3.951435)  (301,14.79885)+-(0,3.963435)  (311,14.80139)+-(0,3.974778)  (321,14.80391)+-(0,3.985429)  (331,14.80646)+-(0,3.995396)  (341,14.80907)+-(0,4.004710)(351,14.81176)+-(0,4.013407)  (361,14.81453)+-(0,4.021522)  (371,14.81736)+-(0,4.029093)  (381,14.82020)+-(0,4.036155)(391,14.82299)+-(0,4.042731)(400,14.83381)+-(0,4.057523) 
                };
                \addplot[color=red,mark=x,error bars/.cd,y dir=both, y explicit] 
                coordinates {
                    (1,11.089125)+-(0,1.323491)(11,9.328351)+-(0,1.155523)(21,8.787697)+-(0,1.235437) (31,8.498729)+-(0,1.290381) (41,8.308991)+-(0,1.326248) (51,8.168557)+-(0,1.348233)
                    (61,8.055765)+-(0,1.359627)(71,7.959818)+-(0,1.363093)  (81,7.875121)+-(0,1.360944)  (91,7.798596)+-(0,1.354897) (101,7.728116)+-(0,1.346278)  (111,7.662274)+-(0,1.336029)  (121,7.600192)+-(0,1.324940)  (131,7.541354)+-(0,1.313663)
                    (141,7.485418)+-(0,1.302690)  (151,7.432187)+-(0,1.292371)  (161,7.381536)+-(0,1.282905)  (171,7.333290)+-(0,1.274326) (181,7.287343)+-(0,1.266599)  (191,7.243536)+-(0,1.259590)  (201,7.201736)+-(0,1.253201)(211,7.161879)+-(0,1.247367) (221,7.123955)+-(0,1.242015)  (231,7.087851)+-(0,1.237022)  (241,7.053486)+-(0,1.232280)  (251,7.020796)+-(0,1.227652)  (261,6.989719)+-(0,1.223004)  (271,6.960173)+-(0,1.218225) (281,6.932083)+-(0,1.213226)  (291,6.905359)+-(0,1.207949)  (301,6.879908)+-(0,1.202409)  (311,6.855613)+-(0,1.196568)  (321,6.832371)+-(0,1.190460)  (331,6.810135)+-(0,1.184144)  (341,6.788831)+-(0,1.177642)(351,6.768409)+-(0,1.171010)  (361,6.748830)+-(0,1.164295)  (371,6.730061)+-(0,1.157524)  (381,6.712097)+-(0,1.150742)  (391,6.694888)+-(0,1.143943)(400,6.680003)+-(0,1.137797) 
                };
                \legend{GD ,RGD}
            \end{axis}
        \end{tikzpicture}
    \end{minipage}

    \vspace{0.5cm}
    
    \begin{minipage}[t]{0.48\linewidth}
        \centering
        \begin{tikzpicture}
            \begin{axis}[
                width=\linewidth, 
                title={$(c): r=3$, $p=20$, $n=2000$, $10t_{2.1}$ noise},
                xlabel = $T$,
                ylabel = {$\text{error}$}, 
                title style={font=\small},
                legend pos=south west, 
                legend columns=2,
                font=\tiny,
                height=6cm,
                xtick={1,50,100,150,200},
                ytick={5.13,6.268375,7.41,12.06885,16.36},
                ymajorgrids=true
            ]
                \addplot[color=blue,mark=+,densely dotted,error bars/.cd, y dir=both, y explicit] 
                coordinates { 
                    (1,13.51766)+-(0,2.162370)(6,12.66623)+-(0,2.416605)(11,12.40324)+-(0,2.713228) (16,12.29167)+-(0,2.947686) (21,12.23398)+-(0,3.127623) (26,12.20162)+-(0,3.267202)(31,12.18199)+-(0,3.379522)
                    (36,12.16986)+-(0,3.472255)  (41,12.16204)+-(0,3.550679)  (46,12.15696)+-(0,3.618178) (51,12.15354)+-(0,3.677250)  (56,12.15118)+-(0,3.729605)  (61,12.14940)+-(0,3.776471)  (66,12.14791)+-(0,3.818692)
                    (71,12.14649)+-(0,3.856892)  (76,12.14500)+-(0,3.891554)  (81,12.14333)+-(0,3.923098)  (86,12.14145)+-(0,3.951897) (91,12.13933)+-(0,3.978296)  (96,12.13697)+-(0,4.002611)  (101,12.134417)+-(0,4.025123)
                    (106,12.13165)+-(0,4.046081)  (111,12.12872)+-(0,4.065700)  (116,12.12565)+-(0,4.084164)  (121,12.12247)+-(0,4.101629)  (126,12.11920)+-(0,4.118226)  (131,12.11586)+-(0,4.134063)  (136,12.11247)+-(0,4.149231) (141,12.10905)+-(0,4.163801)  (146,12.10561)+-(0,4.177831)  (151,12.10216)+-(0,4.191368)  (156,12.09871)+-(0,4.204449)  (161,12.09527)+-(0,4.217099)  (166,12.09184)+-(0,4.229342)  (171,12.08841)+-(0,4.241195)(176,12.08501)+-(0,4.252670)  (181,12.08161)+-(0,4.263780)  (186,12.07823)+-(0,4.274535)  (191,12.07487)+-(0,4.284947)  (196,12.07152)+-(0,4.295025)(200,12.06885)+-(0,4.302856) 
                };
                \addplot[color=red,mark=x, error bars/.cd, y dir=both, y explicit] 
                coordinates {
                    (1,10.507189)+-(0,1.4383435)(6,9.609258)+-(0,0.9888349)(11,9.077971)+-(0,0.9106323) (16,8.705577)+-(0,0.9077156) (21,8.423123)+-(0,0.9245530) (26,8.198401)+-(0,0.9464223)(31,8.013433)+-(0,0.9685479)
                    (36,7.857195)+-(0,0.9893248)  (41,7.722519)+-(0,1.0082342)  (46,7.604526)+-(0,1.0252103) (51,7.499783)+-(0,1.0403466)  (56,7.405812)+-(0,1.0538196)  (61,7.320766)+-(0,1.0657583)  (66,7.243215)+-(0,1.0762870)
                    (71,7.172042)+-(0,1.0855207)  (76,7.106366)+-(0,1.0935869)  (81,7.045476)+-(0,1.1006011)  (86,6.988793)+-(0,1.1066657) (91,6.935828)+-(0,1.1118735)  (96,6.886173)+-(0,1.1163115)  (101,6.839487)+-(0,1.1200581)
                    (106,6.795485)+-(0,1.1232023)  (111,6.753914)+-(0,1.1258108)  (116,6.714559)+-(0,1.1279595)  (121,6.677230)+-(0,1.1297026)  (126,6.641761)+-(0,1.1310957)  (131,6.607996)+-(0,1.1321838)  (136,6.575810)+-(0,1.1330216) (141,6.545074)+-(0,1.1336267)  (146,6.515706)+-(0,1.1340588)  (151,6.487581)+-(0,1.1343093)  (156,6.460664)+-(0,1.1344614)  (161,6.434790)+-(0,1.1344353)  (166,6.410038)+-(0,1.1344107)  (171,6.386104)+-(0,1.1341319)(176,6.363349)+-(0,1.1340671)  (181,6.341054)+-(0,1.1334974)  (186,6.320397)+-(0,1.1337722)  (191,6.299584)+-(0,1.1329503)  (196,6.282076)+-(0,1.1353883)(200,6.268375)+-(0,1.1376101) 
                };
                \legend{GD ,RGD}
            \end{axis}
        \end{tikzpicture}
    \end{minipage}%
    \hfill
    \begin{minipage}[t]{0.48\linewidth}
        \centering
        \begin{tikzpicture}
            \begin{axis}[
                width=\linewidth, 
                title={$(d): r=5$, $p=30$, $n=3000$, $6t_{2.1}$ noise},
                xlabel = $T$,
                ylabel = {$\text{error}$}, 
                title style={font=\small},
                font=\tiny,
                height=6cm,
                xtick={1,50,100,150,200},
                ytick={7.62,8.861090,10.1,12.14232,15.6},
                ymajorgrids=true
            ]
                \addplot[color=blue,mark=+,densely dotted,
                error bars/.cd,
                y dir=both, y explicit] 
                coordinates { 
                    (1,15.94007)+-(0,1.776514)(6,14.98285)+-(0,1.855803)(11,14.48366)+-(0,1.988271) (16,14.13986)+-(0,2.125690) (21,13.88296)+-(0,2.251740) (26,13.68159)+-(0,2.363678)(31,13.51809)+-(0,2.462469)
                    (36,13.38169)+-(0,2.549787)  (41,13.26552)+-(0,2.627253)  (46,13.16501)+-(0,2.696278) (51,13.07699)+-(0,2.758057)  (56,12.99913)+-(0,2.813601)  (61,12.92967)+-(0,2.863766)  (66,12.86722)+-(0,2.909280)
                    (71,12.81071)+-(0,2.950763)  (76,12.75926)+-(0,2.988743)  (81,12.71217)+-(0,3.023676)  (86,12.66886)+-(0,3.055950) (91,12.62885)+-(0,3.085895)  (96,12.59174)+-(0,3.113796)  (101,12.55719)+-(0,3.139891)
                    (106,12.52491)+-(0,3.164384)  (111,12.49466)+-(0,3.187449)  (116,12.46622)+-(0,3.209233)  (121,12.43941)+-(0,3.229862)  (126,12.41406)+-(0,3.249444)  (131,12.39003)+-(0,3.268074)  (136,12.36720)+-(0,3.285835) (141,12.34544)+-(0,3.302800)  (146,12.32466)+-(0,3.319033)  (151,12.30475)+-(0,3.334593)  (156,12.28564)+-(0,3.349533)  (161,12.26726)+-(0,3.363900)  (166,12.24952)+-(0,3.377739)  (171,12.23239)+-(0,3.391088)(176,12.21579)+-(0,3.403985)  (181,12.19969)+-(0,3.416461)  (186,12.18404)+-(0,3.428547)  (191,12.16880)+-(0,3.440270)  (196,12.15395)+-(0,3.451656)(200,12.14232)+-(0,3.460536) 
                };
                \addplot[color=red,mark=x,
                error bars/.cd,
                y dir=both, y explicit] 
                coordinates { 
                    (1,13.641545)+-(0,1.174126)(6,13.194345)+-(0,1.134484)(11,12.843389)+-(0,1.098019) (16,12.543982)+-(0,1.066696) (21,12.278493)+-(0,1.042027) (26,12.038673)+-(0,1.024083)(31,11.819909)+-(0,1.012194)
                    (36,11.619072)+-(0,1.005433)  (41,11.433822)+-(0,1.002827)  (46,11.262253)+-(0,1.003506) (51,11.102791)+-(0,1.006733)  (56,10.954115)+-(0,1.011896)  (61,10.815086)+-(0,1.018514)  (66,10.684747)+-(0,1.026224)
                    (71,10.562274)+-(0,1.034742)  (76,10.446956)+-(0,1.043852)  (81,10.338183)+-(0,1.053358)  (86,10.235414)+-(0,1.063103) (91,10.138155)+-(0,1.072964)  (96,10.045960)+-(0,1.082849)  (101,9.958435)+-(0,1.092679)
                    (106,9.875227)+-(0,1.102394)  (111,9.796020)+-(0,1.111947)  (116,9.720536)+-(0,1.121303)  (121,9.648518)+-(0,1.130430)  (126,9.579729)+-(0,1.139304)  (131,9.513950)+-(0,1.147914)  (136,9.450983)+-(0,1.156250) (141,9.390647)+-(0,1.164305)  (146,9.332773)+-(0,1.172080)  (151,9.277205)+-(0,1.179575)  (156,9.223801)+-(0,1.186791)  (161,9.172434)+-(0,1.193738)  (166,9.122986)+-(0,1.200419)  (171,9.075349)+-(0,1.206840)(176,9.029424)+-(0,1.213006)  (181,8.985121)+-(0,1.218923)  (186,8.942352)+-(0,1.224600)  (191,8.901034)+-(0,1.230042)  (196,8.861090)+-(0,1.235258)(200,8.861090)+-(0,1.239271) 
                };
                \legend{GD ,RGD}
            \end{axis}
        \end{tikzpicture}
    \end{minipage}

    \caption{Comparison of statistical performance between RGD and GD. $T$ and the $\text{error}$ represent the number of iterations and $\left\|\mathcal{A}^{(T)}-\mathcal{A}^{\star}\right\|_{F}$ respectively.}
    \label{fig3}
\end{figure}

\begin{table}[htb]
	\centering
	\caption{Parameter values and initial values in Algorithm \ref{Al1} for Figure \ref{fig3} and the statistical performance of the initializer $\widetilde{\mathcal{A}}$.}
	\label{tab1}
	\resizebox{\textwidth}{!}{
	\begin{tabular}{cccccccc}
		\toprule
		Case & $\lambda$ & $\eta$ & $(a,b)$ & $(\tau,\varpi)$ & Initializer $\widetilde{\mathcal{A}}$ & Failures & Algorithm \\
		\midrule
		\multirow{2}{*}{(a)} & \multirow{2}{*}{5} & \multirow{2}{*}{$1\times 10^{-3}$} & \multirow{2}{*}{$(5,1)$} & $(10,3)\sqrt{n/df}$ & $40.04_{(5.18)}$ & 0 & RGD \\
		& & & & $+\infty$ & $48.92_{(10.53)}$ & 3 & GD \\
		\cmidrule(lr){1-8}
		\multirow{2}{*}{(b)} & \multirow{2}{*}{5} & \multirow{2}{*}{$2\times 10^{-3}$} & \multirow{2}{*}{$(5,1)$} & $(10,5)\sqrt{n/df}$ & $31.84_{(1.93)}$ & 0 & RGD \\
		& & & & $+\infty$ & $55.06_{(7.49)}$ & 14 & GD \\
		\cmidrule(lr){1-8}
		\multirow{2}{*}{(c)} & \multirow{2}{*}{5} & \multirow{2}{*}{$1\times 10^{-3}$} & \multirow{2}{*}{$(5,1)$} & $(15,5)\sqrt{n/df}$ & $46.47_{(5.65)}$ & 0 & RGD \\
		& & & & $+\infty$ & $64.47_{(14.24)}$ & 4 & GD \\
		\cmidrule(lr){1-8}
		\multirow{2}{*}{(d)} & \multirow{2}{*}{8} & \multirow{2}{*}{$8\times 10^{-4}$} & \multirow{2}{*}{$(5,1)$} & $(20,8)\sqrt{n/df}$ & $61.30_{(3.94)}$ & 0 & RGD \\
		& & & & $+\infty$ & $78.83_{(19.01)}$ & 0 & GD \\
		\bottomrule
	\end{tabular}
	}
\end{table}
In the above simulations we considered homogeneous model, i.e., the error distribution is independent of the covariates. In the next experiments, we consider evaluating the performance of RGD under the heteroscedastic model:
\begin{align*}
	y_i=\left\langle\mathcal{X}_{i}, \mathcal{A}^{*}\right\rangle+5c^{-1}\left(\left\langle\mathcal{X}_{i}, \mathcal{A}^{*}\right\rangle\right)^2 \varepsilon_i,
\end{align*}
where the constant $c=\sqrt{3}\left\|\mathcal{A}^{*}\right\|_F^2$ such that $\mathbb{E}\left[c^{-2}\left(\left\langle\mathcal{X}_{i}, \mathcal{A}^{*}\right\rangle\right)^4\right]=1$. In order to model the various shapes of the error distribution, we consider the following three scenarios: (a) $t_{2.1}$ distribution; (b) Pareto distribution; (c) LogNormal distribution. The composition of the tensor parameter $\mathcal{A}^{*}$ is the same as the case of the homogeneous model and considers $r=3$ and $p=13$. We choose $\tau=10\sqrt{n/df}, \varpi=5\sqrt{n/df}$ and $\eta=10^{-1}$. Based on 200 independently repeated experiments, Tables \ref{tab2} and \ref{tab3} show that our method outperforms the least squares method for various heavy-tailed errors, including asymmetric ones.

\begin{table}[htb]
	\centering
	\caption{Simulation results under the heteroskedasticity model. Standard deviations are shown in parentheses.}
	\label{tab2}
	\fontsize{9}{13}\selectfont
	\begin{tabular}{@{}llcccc@{}}
		\toprule
		\multirow{2}{*}{Case} & \multirow{2}{*}{Method} & \multicolumn{4}{c}{$n$} \\
		\cmidrule(lr){3-6}
		     &        & 500 & 1000 & 1500 & 3000 \\
		\midrule
		\multirow{2}{*}{(a)} 
		& RGD & 5.6289 (1.762) & 2.5250 (0.617) & 1.9901 (0.471) & 1.4469 (0.126) \\
		& GD  & 10.5183 (5.459) & 5.8554 (3.227) & 4.3039 (1.805) & 2.9340 (0.830) \\
		\midrule
		\multirow{2}{*}{(b)} 
		& RGD & 4.7553 (2.107) & 2.1107 (0.767) & 1.4920 (0.294) & 1.0559 (0.088) \\
		& GD  & 7.8253 (5.238) & 4.6653 (3.478) & 3.6750 (2.737) & 2.2123 (1.162) \\
		\midrule
		\multirow{2}{*}{(c)} 
		& RGD & 5.3948 (1.812) & 2.5135 (0.653) & 1.8884 (0.332) & 1.3708 (0.127) \\
		& GD  & 9.0962 (4.277) & 4.7237 (2.025) & 3.6291 (1.178) & 2.4001 (0.475) \\
		\bottomrule
	\end{tabular}
\end{table}

\begin{table}[htb]
	\centering
	\caption{The number of times the algorithm failed to converge in the experiments of Table \ref{tab2}.}
	\label{tab3}
	\fontsize{9}{11}\selectfont
	\begin{tabular}{@{}llcccc@{}}
		\toprule
		\multirow{2}{*}{Case} & \multirow{2}{*}{Method} & \multicolumn{4}{c}{$n$} \\
		\cmidrule(lr){3-6}
		     &        & 500 & 1000 & 1500 & 3000 \\
		\midrule
		\multirow{2}{*}{(a)} 
		& RGD & 0 & 0 & 0 & 0 \\
		& GD  & 17 & 10 & 12 & 10 \\
		\midrule
		\multirow{2}{*}{(b)} 
		& RGD & 0 & 0 & 0 & 0 \\
		& GD  & 11 & 3 & 11 & 8 \\
		\midrule
		\multirow{2}{*}{(c)} 
		& RGD & 0 & 0 & 0 & 0 \\
		& GD  & 10 & 10 & 10 & 9 \\
		\bottomrule
	\end{tabular}
\end{table}

\subsection{Phase transition phenomenon}
According to Remark \ref{remark4}, the resulting estimator has the theoretical rate of order  $C(M_{\delta},\kappa)\times\left(\frac{df}{n}\right)^{\min\left\{\frac{1}{2},\frac{\delta}{1+\delta}\right\}}$  under Frobenius norm, where $C(M_{\delta},\kappa)$ is a positive constant depending only on $M_{\delta}$ and $\kappa$. This means that when $\delta<1$, the statistical error will increase sharply. To verify this phase transition phenomenon, we choose scaled Student's $t_{\nu}$-distributions with degrees of freedom $\nu\in \{1.01,1.1,1.2,\ldots,2.9,3\}$ as the error distribution. Take $\delta=\nu-1.01$ in the selection of truncation and robustification parameters when $\nu\leq 2$, otherwise $\delta=1$. The statistical behavior of the adaptive Huber estimator is shown in Figure \ref{fig4}. For simplicity, the constant $C(M_{\delta},\kappa)$ in the theoretical bound is set as the fixed constant 13.4 for (a) and 13.2 for (b). The empirical curves in Figure \ref{fig4} match the theoretical curves very well when $\delta\leq 1$. When $\delta> 1$, the statistical error still decreases gradually as $\delta$ increases, which is consistent with the theory and intuitive expectation. Specifically, it is general knowledge that $t_{\nu}\to \mathcal{N}(0,1)$ as $\nu\to \infty$ and $M=\mathbb{E}[t_{\nu}^2]=\frac{\nu}{\nu-2}\stackrel{\nu\to \infty}{\searrow} 1$. When the order of the moment grows, $C(M,\kappa)$ will drop and the tail of the noise distribution is lighter. Therefore, the proposed estimator achieves a better statistical performance as expected than the case of the relatively small $\delta$. 
\begin{figure}[htbp]
    \centering
    \begin{minipage}[t]{0.48\textwidth}
        \centering
        \begin{tikzpicture}
            \begin{axis}[
                width=\linewidth, 
                title={$(a): r=3, p=10, n=300, 5t_{\nu}$ noise},
                legend columns=1,
                title style={font=\small},
                xlabel = $\delta$,
                ylabel = {$\|\mathcal{A}^{(T_{\max})}-\mathcal{A}^{\star}\|_{F}$},
                legend style={at={(0.98,0.98)}, anchor=north east, font=\tiny}, 
                xtick={0,0.5,1,1.5,2},
                font=\footnotesize,
                height=6.5cm, 
                ymajorgrids=true
            ]
                \addplot[smooth,color=red,mark=o]
                coordinates {
                    (0,13.6633)(0.1-0.01,12.47206)(0.2-0.01,11.55077)(0.3-0.01,10.99128)(0.4-0.01,10.46461)(0.5-0.01,9.998508)(0.6-0.01,9.619845)(0.7-0.01,9.059196)(0.8-0.01,8.756197)(0.9-0.01,8.720594
                    )(1-0.01,8.390624)(1.1-0.01,8.363561)(1.2-0.01,8.317641)(1.3-0.01,8.011802)(1.4-0.01,7.970411)(1.5-0.01,7.897793)(1.6-0.01,7.874986)(1.7-0.01,7.99651)(1.8-0.01,7.884036)(1.9-0.01,7.978035)(2-0.01,7.966126) 
                };
                \addplot [domain=-0.03:2, samples=300, color=blue]
                {13.4*((27+3*10*3)/300)^(min(0.5,x/(1+x)))};
                \legend{RGD,Theoretical Bound}
            \end{axis}
        \end{tikzpicture}
    \end{minipage}%
    \hfill
    \begin{minipage}[t]{0.48\textwidth}
        \centering
        \begin{tikzpicture}
            \begin{axis}[
                width=\linewidth,
                title={$(b): r=3, p=15, n=500, 5t_{\nu}$ noise},
                legend columns=1,
                title style={font=\small},
                xlabel = $\delta$,
                ylabel = {$\|\mathcal{A}^{(T_{\max})}-\mathcal{A}^{\star}\|_{F}$},
                legend style={at={(0.98,0.98)}, anchor=north east, font=\tiny},
                xtick={0,0.5,1,1.5,2},
                ytick={6.8,8,10,12,14},
                font=\footnotesize,
                height=6.5cm, 
                ymajorgrids=true
            ]
                \addplot[smooth,color=red,mark=o]
                coordinates {
                    (0,13.60921)(0.1-0.01,12.31269)(0.2-0.01,11.08395)(0.3-0.01,10.09575)(0.4-0.01,9.712392)(0.5-0.01,9.14836)(0.6-0.01,8.83104)(0.7-0.01,8.613623)(0.8-0.01,8.132673)(0.9-0.01,7.71228)(1-0.01,7.517182)(1.1-0.01,7.401869)(1.2-0.01,7.333676)(1.3-0.01,7.326894)(1.4-0.01,7.353094)(1.5-0.01,7.074537)(1.6-0.01,7.16725)(1.7-0.01,6.931935)(1.8-0.01,7.011848)(1.9-0.01,6.703529)(2-0.01,6.927234) 
                };
                \addplot [domain=-0.03:2, samples=300, color=blue]
                {13.2*((27+3*15*3)/500)^(min(0.5,x/(1+x)))};
                \legend{RGD,Theoretical Bound}
            \end{axis}
        \end{tikzpicture}
    \end{minipage}
    
    \caption{The trend of statistical errors with varying $\delta$ where $T_{\max}=500$.}
    \label{fig4}
\end{figure}

\subsection{Image recovery}
	Since a two-dimensional image can be considered as a matrix, one approach to image compression is to approximate it by singular value decomposition using a low-rank matrix \citep[see][]{Recht2010,Wakin2006}. In this section, we show the application of our algorithm in image restoration, which is used to verify the feasibility and effectiveness of the proposed method. Firstly, we consider the matrix version of (\ref{eq}), i.e., the matrix compressed sensing model:
\begin{align*}
	y_i=\left\langle X_{i}, A^{*}\right\rangle+\varepsilon_i,
\end{align*}
where $X_{i}\in \mathbb{R}^{d_{1}\times d_{2}}$ are covariates, $A^{*}\in \mathbb{R}^{d_{1}\times d_{2}}$ is a low-rank matrix parameter and $\text{rank}(A^{*})=r$. Write $A^{*}$ in the form of singular value decomposition denoted as $A^{*}=U^{*}S^{*}V^{*\top}$ and obtain the low-rank estimate of $A^{*}$ by solving the following optimization problem:
\begin{align*}
	\left(\widehat{S}, \widehat{U}, \widehat{V}\right)=\underset{S, U, V}{\arg \min }\left\{\mathcal{L}_\varpi\left(U S V^{\top}\right)+\frac{a}{2}\left(\left\|U^{\top} U-I_r\right\|_F^2+\left\|V^{\top} V-I_r\right\|_F^2\right)\right\}.
\end{align*}
By the matrix version of Algorithm \ref{Al1} and Theorem \ref{theorem1}, it can be obtained that when $n \gtrsim R^2 \bar{p} r$, after enough iterations $T_{\max}$, we have with high probability
\begin{align*}
	\left\|A^{\left(T_{\max }\right)}-A^{*}\right\|_F \lesssim \kappa^2 \sqrt{M \bar{p} \bar{r} / n},
\end{align*}
where $R:=\mathbb{E}\left[y_i^2\right], \kappa:=\left\|A^{*}\right\|_{\mathrm{op}} / \sigma_r\left(A^{*}\right), \bar{p}:=d_1 \vee d_2,\bar{r}:=r$ and $ M:=\sqrt[k]{\mathbb{E}\left(\mathbb{E}\left(\varepsilon_i^2 \mid X_i\right)^k\right)}<\infty$. This is the same as error upper bound established by \cite{Negahban2011} in sub-Gaussian noise. Next, we use each of three $43\times 53$ dimensional 0-1 matrices (Figure 1 of \cite{Kong2020}) as the parameter matrices of (\ref{eq}). As shown in the first row of images in Figure \ref{fig5}, we denote them as $\{A^{*}_{i}\}_{i\in [3]}$  respectively. Each term of the covariates is chosen to be i.i.d. samples from $\mathcal{N}(0,1)$. For the random error term, we consider three levels of scaled Student's $t$-distributions. Specifically, the noise is generated as $\varepsilon_i = C \cdot T$, where $T \sim t_{2.1}$ is a random variable following a Student's t-distribution with 2.1 degrees of freedom. The scaling factor $C$ is set to 1, 2, and 4, respectively, in each of the three experiments.

Table \ref{tab4} shows the simulation results based on 200 repeated experiments, demonstrating that the proposed robust estimator has better statistical performance than the benchmark in terms of mean and standard deviation. We randomly select a dataset from the 200 experiments and plot the estimators of $\{A^{*}_{i}\}_{i\in [3]}$. These images are shown in Figure \ref{fig5}, where the images in the second row are the images reconstructed by the estimator proposed in this paper; The third row is the image reconstructed by the original least squares method. Figure \ref{fig5} illustrates that our robust estimator outperforms the conventional least squares estimation.

Secondly, we confirm the superiority of the algorithm under heavy-tailed noise by combining each channel of the four cigarette package images (Golden Dragon, Yellow Crane Tower,Hatamen and Hongtashan), each of which has the resolution of $110\times 70$ into a third-order tensor with appropriately scaling as a parameter of the model (i.e., $p_{1}=110,p_{2}=70,p_{3}=12$) (shown in Figure \ref{fig6}).

Each term of the covariates is chosen to be i.i.d.following $\mathcal{N}(0,1)$, and the random error term $\epsilon_{i}$ is generated as $\varepsilon_i = 5 \cdot T$, where $T \sim t_{2.1}$ (a student's $t$-distribution with 2.1 degrees of freedom). The sample size is $n=25000$. We choose $r_{1}=60,r_{2}=40,r_{3}=12$ as the rank of the recovered tensor and the parameters are selected as $a=5$,$b=1$ and $\varpi=100\sqrt{n/df}$.  The error $\left\|A^{\left(T_{\max }\right)}-A^{*}\right\|_F$ obtained after $T_{\max}=500$ iterations is 15.1514. The recovered parameter images as shown in Figure \ref{fig7}. The results show that Algorithm \ref{Al1} can better recover the original image from its compressed measurements. For the least squares method, the algorithm will not converge.

     \begin{figure}[htbp]
    	\centering
    	\begin{minipage}{0.32\linewidth}
    		\includegraphics[width=\linewidth]{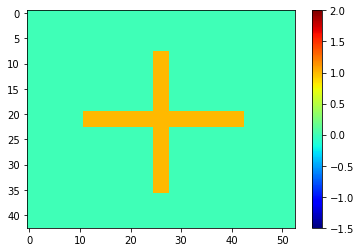}
    	\end{minipage}
    	\hfill
    	\begin{minipage}{0.32\linewidth}
    		\includegraphics[width=\linewidth]{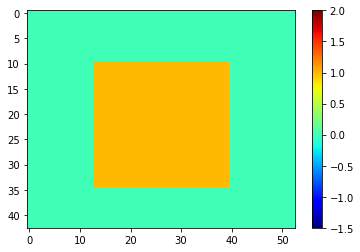}
    	\end{minipage}
    	\hfill
    	\begin{minipage}{0.32\linewidth}
    		\includegraphics[width=\linewidth]{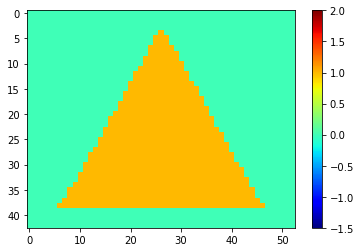}
    	\end{minipage}
    	\centering
    	\begin{minipage}{0.32\linewidth}
    		\includegraphics[width=\linewidth]{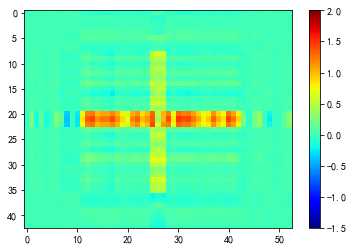}
    	\end{minipage}
    	\hfill
    	\begin{minipage}{0.32\linewidth}
    		\includegraphics[width=\linewidth]{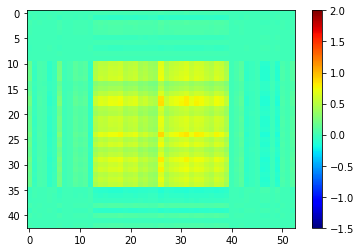}
    	\end{minipage}
    	\hfill
    	\begin{minipage}{0.32\linewidth}
    		\includegraphics[width=\linewidth]{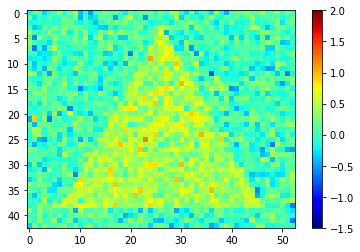}
    	\end{minipage}
    	\centering
    	\begin{minipage}{0.32\linewidth}
    		\includegraphics[width=\linewidth]{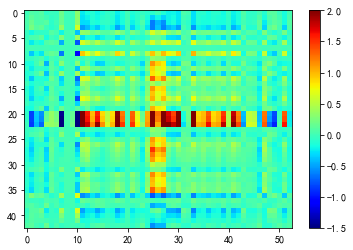}
    	\end{minipage}
    	\hfill
    	\begin{minipage}{0.33\linewidth}
    		\includegraphics[width=\linewidth]{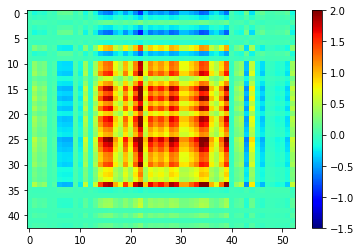}
    	\end{minipage}
    	\hfill
    	\begin{minipage}{0.33\linewidth}
    		\includegraphics[width=\linewidth]{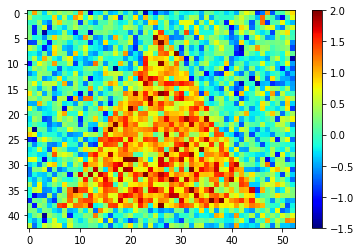}
    	\end{minipage}
    	\setlength{\abovecaptionskip}{0.3cm}
    	\caption{Images constructed by randomly selecting an experimental result from Table \ref{tab4}. }
    	\label{fig5}
    \end{figure}

\begin{table}[htb]
	\centering
	\caption{Comparison of estimation errors for two methods. Standard deviations are shown in parentheses.}
	\label{tab4}
	\fontsize{11}{14}\selectfont
	\begin{tabular}{@{}lccc@{}}
		\toprule
		Method & $A^{*}_{1}$ ($n=5000$) & $A^{*}_{2}$ ($n=10000$) & $A^{*}_{3}$ ($n=10000$) \\
		\midrule
		RGD & $0.6045\;(0.0295)$ & $0.9159\;(0.0127)$ & $0.6647\;(0.0073)$ \\
		GD  & $1.1600\;(0.8077)$ & $1.2499\;(0.1514)$ & $0.8682\;(0.0187)$ \\
		\bottomrule
	\end{tabular}
\end{table}
   
\begin{figure}[htbp]
    	\centering
    	\begin{minipage}{0.22\linewidth}
    		\includegraphics[width=\linewidth]{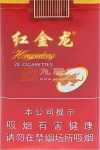}
    	\end{minipage}
    	\hfill
    	\begin{minipage}{0.22\linewidth}
    		\includegraphics[width=\linewidth]{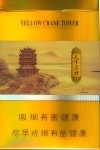}
    	\end{minipage}
    	\hfill
    	\begin{minipage}{0.22\linewidth}
    		\includegraphics[width=\linewidth]{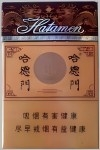}
    	\end{minipage}
    	\hfill
    	\begin{minipage}{0.22\linewidth}
    		\includegraphics[width=\linewidth]{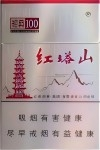}
    	\end{minipage}
    	\setlength{\abovecaptionskip}{0.3cm}
    	\caption{The images of cigarette package.}
    	\label{fig6}
    \end{figure}
    
\begin{figure}[htbp]
	\centering
	\begin{minipage}{0.22\linewidth}
		\includegraphics[width=\linewidth]{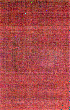}
	\end{minipage}
	\hfill
    \begin{minipage}{0.22\linewidth}
    	\includegraphics[width=\linewidth]{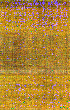}
    \end{minipage}
    \hfill
    \begin{minipage}{0.22\linewidth}
    	\includegraphics[width=\linewidth]{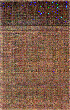}
    \end{minipage}
    \hfill
    \begin{minipage}{0.22\linewidth}
    	\includegraphics[width=\linewidth]{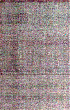}
    \end{minipage}
    \setlength{\abovecaptionskip}{0.3cm}
    \caption{Recovered images.}
    \label{fig7}
\end{figure}

\subsection{An application to Beijing air-quality dataset}
We use the Beijing Air Quality dataset\footnote{The Beijing Air Quality dataset \citep{Du2021} can be accessed through the UCI repository at: \url{https://archive.ics.uci.edu/dataset/501/beijing+multi+site+air+quality+data}.} to predict daily mean of PM2.5 concentrations and verify the superiority of our method. The feature matrix $X_i$ encompasses 24-hourly measurements of eight covariates: the concentrations of $\mathrm{SO_{2}}$, $\mathrm{NO_2}$, CO, and $\mathrm{O_3}$, along with the meteorological variables of temperature, atmospheric pressure, dew point, and wind speed. Thus, $X_i \in \mathbb{R}^{24 \times 8}$. The corresponding response variable $y_i \in \mathbb{R}$ is the daily mean $\mathrm{PM_{2.5}}$ concentration. To apply our tensor regression framework to this problem, we treat the feature matrices $X_i$ as 2nd-order tensors. Consequently, our general model specializes to the matrix regression setting, for which we aim to estimate a low-rank coefficient matrix $A^* \in \mathbb{R}^{24 \times 8}$. The complete dataset is collected from 12 monitoring stations. Each station contributes 1461 such daily observations $(X_i, y_i)$. In preprocessing, missing values in the features were imputed using their respective column means, after which all features were normalized.

We utilized data from the top five stations for testing our algorithm, assigning the first 1,200 samples to the training set and the subsequent 261 samples to the test set. We choose $r=5$ and select robustification parameters $\tau$ and $\varpi$ by 5-fold cross-validation. We use RMSE to measure the performance of our method in test set, which is defined as
\begin{align*}
	\mathrm{RMSE}=\left\{\frac{1}{n_{\text {test}}} \sum_{i=1}^{n_{\text {test}}}\left(y_i-\widehat{y}_i\right)^2\right\}^{1 / 2}.
\end{align*}
The prediction results of our robust method (RGD) are shown in Figure \ref{fig8}. Notably, the standard gradient descent (GD) algorithm failed to converge on this dataset, and its results are therefore omitted. This failure is likely attributable to the presence of outliers or heavy-tailed distributions inherent in the real-world air quality data, which underscores the practical necessity of our robust methodology.

\begin{figure}[htbp]
   	\centering
   	\begin{tikzpicture}
   		\begin{axis}
   			[ybar,
   			grid=major,major grid style={dashed},
   			ymin=0,
   			ymax=0.6,
   			ylabel=RMSE,
   			bar width=.4cm,
   			width=\linewidth,
   			height=6cm,
   			symbolic x coords={Aotizhongxin, Changping, Dingling, Dongsi,Guanyuan},
   			xtick=data,
   			axis y line=left, 
   			axis x line=bottom, 
   			nodes near coords,
   			nodes near coords style={font=\fontsize{8}{12}\selectfont},
   			enlarge x limits=0.15,
   			legend style={at={(0.5,-0.2)},anchor=north,legend columns=-1},
   			] 
   			\addplot+ coordinates {(Aotizhongxin, 0.5532) (Changping, 0.5501) (Dingling,0.5360) (Dongsi, 0.5056) (Guanyuan, 0.55528) }; 
   		\end{axis}
   	\end{tikzpicture}
\caption{The prediction error of Algorithm \ref{Al1}.}
\label{fig8}
\end{figure}

\section{Discussion}\label{discussion}
This paper proposes a robust parameter estimation framework for low-rank tensor regression model in the context of heavy-tailed errors, and theoretically establishes optimal convergence guarantee for only finite second-order moment noise. We also conduct simulations under both homogeneous and heterogeneous models, demonstrating that the proposed robust estimator outperforms traditional methods in both scenarios. Finally, the method is applied to image recovery and Beijing air-quality dataset, yielding significant results. Furthermore, while this study focuses on parameter estimation for third-order tensor regression, our algorithm and results can also be extended to $d$-order low-rank tensor regression models. Specifically, 
\begin{align*}
& \left(\widehat{\mathcal{S}}, \widehat{\mathbf{U}}_1, \widehat{\mathbf{U}}_2, \cdots, \widehat{\mathbf{U}}_d\right) \\
={}& \underset{\mathcal{S}, \mathbf{U}_1, \mathbf{U}_2, \cdots, \mathbf{U}_d}{\arg \min }\left\{\mathcal{L}_{\varpi}\left(\llbracket \mathcal{S} ; \mathbf{U}_1, \mathbf{U}_2, \cdots, \mathbf{U}_d \rrbracket\right)+\frac{a}{2} \sum_{k=1}^d\left\|\mathbf{U}_k^{\top} \mathbf{U}_k-b^2 \mathbf{I}_{r_k}\right\|_F^2\right\}.
\end{align*}

Under conditions of Theorem \ref{theorem1}, by choosing $\varpi\asymp_{K,k}\left(Mn/df\right)^{\frac{1}{2}}$ and $\tau\asymp_{K,k}\left(Rn/df\right)^{\frac{1}{2}}$, after a sufficiently large number of iterations and sample sizes, the robust estimator yielded by Algorithm \ref{Al1} has the error upper bound of 
\begin{align*}
	\left\|\mathcal{A}^{(T_{\max})}-\mathcal{A}^{*}\right\|_{F}\lesssim \kappa^{2}\sqrt{Mdf/n}
\end{align*}
with high probability, where the degrees of freedom are generalized to $df := \prod_{k=1}^d r_k + \sum_{k=1}^d p_k r_k$. The terms $\bar{p}:=\max\{p_{k}:k\in [d]\}$, $\bar{r}:=\max\{r_{k}:k\in [d]\}$ and $\kappa$ are generalized accordingly from the 3rd-order case. On the other hand, our improved algorithm can also be used in the low SNR case (i.e., the case where $\left\|\mathcal{A}^{*}\right\|_F^2/ \mathbb{E}\left[\varepsilon_i^2\right]$ is low). Due to space limitation, we do not continue the simulation demonstration in this paper.

There are still many areas in this paper that can be further improved and studied. The robust parameter $\varpi$ of Huber loss in the paper can be determined through cross-validation, while \cite{Wang2021} proposed a data-driven adaptive equation to determine the magnitude of the robust parameter in high-dimensional sparse linear regression model. Whether a similar adaptive method for solving the robustification parameter can be designed for low-rank tensor regression model is a question for future research. On the other hand, the assumption in Theorem \ref{theorem1} that each covariate term is i.i.d. is difficult to satisfy in practice. Whether this condition can be weakened is a very worthwhile research question. 

\section*{Declaration of competing interest}
The authors declare that they have no known competing financial interests or personal relationships that could have appeared to influence the work reported in this paper.

\section*{Funding}
This research did not receive any specific grant from funding agencies in the public, commercial, or not-for-profit sectors.



\appendix
\section{Proof of Theorem \ref{theorem1}}
\label{sec:proof1}
For the convenience of the proof, we use the following mathematical notation: Since the loss function $\mathcal{L}_{\varpi}(\mathcal{A})=\frac{1}{n}\sum_{i=1}^{n}\ell_{\varpi}(y_{i}-\left\langle\mathcal{X}_{i}, \mathcal{A}\right\rangle)$, denote $\mathcal{X}: \mathbb{R}^{p_1 \times p_2 \times p_3} \rightarrow \mathbb{R}^n$ as the linear operator such that $[\mathcal{X}(\mathcal{A})]_i=$ $\left\langle\mathcal{X}_i, \mathcal{A}\right\rangle$ and $\mathcal{X}^*$ is denoted as the adjoint operator of $\mathcal{X}$:
\begin{align*}
\mathcal{X}^*(x)=\frac{1}{n}\sum_{j=1}^n x_j \mathcal{X}_j, \quad x \in \mathbb{R}^n.	
\end{align*}
Then equation (\ref{eq}) can be rewritten as $y=\mathcal{X}\left(\mathcal{A}^*\right)+\varepsilon \in \mathbb{R}^n,  \varepsilon:=(\varepsilon_{1}, \ldots, \varepsilon_{n})^{\top}$ and the gradient of the loss function $\nabla\mathcal{L}_{\varpi}(\mathcal{A})$ has the analytic form $\mathcal{X}^*\left(\psi_{\varpi}\left(y-\mathcal{X}(\mathcal{A})\right)\right)$. For each iteration $t=0,1,\ldots, T_{\max}$, define the measurement error:
\begin{align*}
E^{(t)}:=\min _{\substack{\mathbf{R}_k \in \mathbb{O}^{p_k\times r_k} \\ k=1,2,3}}\left\{\sum_{k=1}^3\left\|\mathbf{U}_k^{(t)}-\mathbf{U}_k^* \mathbf{R}_k\right\|_{F}^2+\left\|\mathcal{S}^{(t)}-\llbracket \mathcal{S}^* ; \mathbf{R}_1^{\top}, \mathbf{R}_2^{\top}, \mathbf{R}_3^{\top} \rrbracket\right\|_{F}^2\right\}.
\end{align*}
With the above notation, we proceed to prove the theorem in the following two steps:
 
Step (1): $E^{(0)} \lesssim \frac{\underline{\lambda}^{1 / 2}}{\kappa^{3 / 2}}$;

Step (2): The loss function $\mathcal{L}_{\varpi}(\mathcal{A}^{(T)})$ satisfies the constrained strong convexity condition with high probability within the closed set centered at $\mathcal{A}^*$.

Proof of step (1): According to the proof of Theorem 4.1 in \cite{Han2022}, we have
\begin{align*}
	& \left\|\mathcal{A}^{(0)}-\mathcal{A}^{*}\right\|_{F} \\
	\leq{}& \left\|\mathcal{A}^{*} \times_{1}\mathbb{P}_{\widetilde{\mathbf{U}}_{1}} \times_{2} \mathbb{P}_{\widetilde{\mathbf{U}}_{2}}\times_{3}\mathbb{P}_{\widetilde{\mathbf{U}}_{3}}-\mathcal{A}^{*}\right\|_{F}+\left\|\left(\widetilde{\mathcal{A}}-\mathcal{A}^{*}\right)\times_{1}\mathbb{P}_{\widetilde{\mathbf{U}}_{1}}\times_{2}\mathbb{P}_{\widetilde{\mathbf{U}}_{2}}\times_{3}\mathbb{P}_{\widetilde{\mathbf{U}}_{3}}\right\|_{F} \\
	\leq{}& \sum_{k=1}^{3}\left\|\mathcal{M}_{k}(\mathcal{A}^{*})\right\|_{\text{op}}\cdot\left\|\sin\Theta(\widetilde{\mathbf{U}}_{k},\mathbf{U}_{k})\right\|_{F}+\sup_{\substack{\mathcal{T} \in \mathbb{R}^{p_{1} \times p_{2} \times p_{3}}\\ \|\mathcal{T}\|_{\mathrm{F}} \leq 1}}\left\langle\widetilde{\mathcal{A}}-\mathcal{A}^{*},\llbracket\mathcal{T};\widetilde{\mathbf{U}}_{1}^{\top},\widetilde{\mathbf{U}}_{2}^{\top}, \widetilde{\mathbf{U}}_{3}^{\top}\rrbracket\right\rangle \\
	\leq{}& \bar{\lambda}\sum_{k=1}^3\sqrt{r_{k}}\left(\frac{\sigma_{r_{k}}\left(\mathbf{U}_{k}^{*\top} \mathcal{M}_{k}(\widetilde{\mathcal{A}})\right)\left\|\mathbf{U}_{k,\perp}^{*\top} \mathcal{M}_{k}(\widetilde{\mathcal{A}})\mathbb{P}_{(\mathbf{U}_{k}^{*\top} \mathcal{M}_{k}(\widetilde{\mathcal{A}}))^{\top}}\right\|_{\text{op}}}{\sigma_{r_{k}}^2\left(\mathbf{U}_{k}^{\top}\mathcal{M}_{k}(\widetilde{\mathcal{A}})\right)-\sigma_{r_{k}+1}^2\left(\mathcal{M}_{k}(\widetilde{\mathcal{A}})\right)}\wedge 1\right) \\
	&+ \sup_{\substack{\mathcal{T}\in \mathbb{R}^{p_{1} \times p_{2} \times p_{3}}, \|\mathcal{T}\|_{\mathrm{F}}\leq1\\ \text{rank}(\mathcal{T})\leq(r_{1},r_{2},r_{3})}}\left\langle\widetilde{\mathcal{A}}-\mathcal{A}^{*},\mathcal{T}\right\rangle=: T_{1}+T_{2},
\end{align*}
where the last inequality follows from Proposition 1 of \cite{Cai2018}. We first find the upper bound of $T_{2}$: For a fixed $\mathcal{T}$ with $\|\mathcal{T}\|_{F}\leq 1$,
\begin{align*}
\left|\left\langle\widetilde{\mathcal{A}}-\mathcal{A}^{*},\mathcal{T}\right\rangle\right|&=\left|\sum_{(j,k,l)}\left(\widetilde{\mathcal{A}}_{(i,j,k)}-\mathcal{A}^{*}_{(i,j,k)}\right)\mathcal{T}_{(j,k,l)}\right|=\left|\frac{1}{n}\sum_{i=1}^{n}\psi_{\tau}(y_{i})z_{i}-\mathbb{E}\left[y_{i}z_{i}\right]\right|\\&\leq \left|\frac{1}{n}\sum_{i=1}^{n}\psi_{\tau}(y_{i})z_{i}-\mathbb{E}\left[\psi_{\tau}(y_{i})z_{i}\right]\right|	+\left|\mathbb{E}\left[\psi_{\tau}(y_{i})z_{i}\right]-\mathbb{E}\left[y_{i}z_{i}\right]\right|,
\end{align*}
where $z_{i}:=\sum_{(j,k,l)}\mathcal{X}_{i(j,k,l)}\mathcal{T}_{(j,k,l)}$. By the rotation invariance of sub-Gaussian random variable, we get that $z_{i}$ is also a sub-Gaussian random variable with $\|z_{i}\|_{\psi_{2}}=K$. Therefore, since $\sqrt[k]{\mathbb{E}|z_{i}|^k}\leq \sqrt{k}K$ for all $k\geq1$ and $\|\text{vec}(\mathcal{X}_{i})\|_{\psi_{2}}=K$, we have for $q\geq2$,
\begin{align}
		&\mathbb{E}\left|\psi_{\tau}(y_{i})z_{i}\right|^q\leq \tau^{q-2}\mathbb{E}|y_{i}^{2}z_{i}^{q}|\leq 2\cdot\tau^{q-2}\left(\mathbb{E}|\varepsilon_{i}|^{2}|z_{i}^q|+\mathbb{E}\left|\left\langle\mathcal{X}_{i}, \mathcal{A}^{*}\right\rangle\right|^{2} |z_{i}^q|\right) \notag\\
		\leq{}& 2\tau^{q-2}\left(\sqrt[k]{\mathbb{E}\left(\mathbb{E}\left(|\varepsilon_{i}|^{2}|\mathcal{X}_{i}\right)^k\right)}\left(\mathbb{E}|z_{i}^{\frac{qk}{k-1}}|\right)^{\frac{k-1}{k}}+\sqrt{\mathbb{E}\left|\left\langle\mathcal{X}_{i}, \mathcal{A}^{*}\right\rangle\right|^{4}\mathbb{E}|z_{i}^{2q}|}\right) \notag \\
		\leq{}& 2\tau^{q-2}\left(M\left(\frac{K^2qk}{k-1}\right)^{q/2}+4K^{2}\|\mathcal{A}^{*}\|_{F}^{2}\left(K\sqrt{2q}\right)^q\right) \notag \\
		\leq{}& 2q! M\left(eK\tau\sqrt{\frac{2 k}{k-1}}\right)^{q-2}, \label{eq1}
\end{align}
where the last inequality is derived from Stirling's inequality and $\|\mathcal{A}^{*}\|_{F}^2\leq \mathbb{E}\left|\left\langle\mathcal{X}_{i}, \mathcal{A}^{*}\right\rangle\right|^2\leq \mathbb{E}[y_{i}^2]\leq R$. By Bernstein's inequality, we obtain that there exists an absolute constant $C$ depending on $K, k$ and $c$ such that
\begin{equation}\label{eq2}
	\operatorname{P}\left(\left|\frac{1}{n}\sum_{i=1}^{n}\psi_{\tau}(y_{i})z_{i} - \mathbb{E}\psi_{\tau}(y_{i})z_{i}\right|\leq C\sqrt{\frac{R t}{n}}+C\frac{\tau t}{n}\right) \geq 1-2e^{-t}.
\end{equation}
For the second term, we have
\begin{align}\label{eq3}
& \left|\mathbb{E}\left[\psi_{\tau}(y_{i})z_{i}\right] - \mathbb{E}\left[y_{i}z_{i}\right]\right| = \left|\mathbb{E}\left[(\psi_{\tau}(y_{i})-y_{i})z_{i}\right]\right| \notag \\
={} & \left|\mathbb{E}\left[(\tau \operatorname{sign}(y_{i}) - y_{i}) 1_{\{|y_{i}|>\tau\}} z_{i}\right]\right| \notag \\
\le{} & \mathbb{E}\left[|y_{i}| 1_{\{|y_{i}|>\tau\}} |z_{i}|\right] + \tau \mathbb{E}\left[1_{\{|y_{i}|>\tau\}} |z_{i}|\right] \notag \\
\le{} & \sqrt{\mathbb{E}[y_i^2 z_i^2] \operatorname{P}(|y_i| > \tau)} + \tau \sqrt{\mathbb{E}[z_i^2] \operatorname{P}(|y_i| > \tau)} \notag \\
\le{} & \frac{\sqrt{\mathbb{E}[y_i^2 z_i^2] \mathbb{E}[y_i^2]}}{\tau} + \frac{\tau \sqrt{\mathbb{E}[z_i^2]\mathbb{E}[y_i^2]}}{\tau} \notag \\
\lesssim{} & \frac{\sqrt{R \cdot R}}{\tau} + \sqrt{1 \cdot R} \lesssim R/\tau.
\end{align}
Therefore, by combining (\ref{eq2}) with (\ref{eq3}), there is a constant $C_{1}$ depending only on $K, k$ and $c$ such that
\begin{align*}
	\operatorname{P}\left(\left|\left\langle\widetilde{\mathcal{A}}-\mathcal{A}^{*},\mathcal{T}\right\rangle\right|\leq C_{1}\sqrt{\frac{Rt}{n}}+C_{1}\frac{\tau t}{n}+C_{1}R
		/\tau\right)\geq  1-2e^{-t}.
\end{align*}
By following the $\frac{1}{3}$-net argument in Lemma E.5 of \cite{Han2022}, we can obtain that with probability at least $1-2\exp\left(df(\log(13)-t)\right),$
\begin{align*}
	\sup_{\substack{\mathcal{T}\in \mathbb{R}^{p_{1} \times p_{2} \times p_{3}}, \|\mathcal{T}\|_{\mathrm{F}}\leq1\\ \text{rank}(\mathcal{T})\leq(r_{1},r_{2},r_{3})}}
		\left\langle\widetilde{\mathcal{A}}-\mathcal{A}^{*},\mathcal{T}\right\rangle\leq C\sqrt{\frac{M df\times t}{n}}+C\frac{\tau df\times t}{n}+C_{1}M/\tau,
\end{align*}
where $C_{2}$ is an absolute constant. By choosing $\tau\asymp_{K,k}\left(Mn/df\right)^{\frac{1}{2}}$ and  $t>\log(13)$, the above inequality changes into
\begin{equation}\label{eq4}
	\operatorname{P}\left(\sup_{\substack{\mathcal{T}\in \mathbb{R}^{p_{1} \times p_{2} \times p_{3}}, \|\mathcal{T}\|_{\mathrm{F}}\leq1\\ \text{rank}(\mathcal{T})\leq(r_{1},r_{2},r_{3})}}\left\langle\widetilde{\mathcal{A}}-\mathcal{A}^{*},\mathcal{T}\right\rangle\leq C_{3}\sqrt{\frac{Rdf}{n}}\right)\geq 1-2\exp\left(-C_4 df\right).
\end{equation}
		Next, we solve for the upper bound of $T_{1}$. In the proof of Theorem 4 in \cite{Zhang2020}, via applying Lemma \ref{lemma2} and Lemma \ref{lemma3}, we obtain that there exists a positive constant $c<1$ such that with probability at least $1-2 e^{-c p_{1}}-4e^{-t}$, the following inequality holds: 
\begin{equation}\label{eq6}
	\begin{aligned}
		&\quad\frac{\sigma_{r_{k}}\left(\mathbf{U}_{k}^{*\top} \mathcal{M}_{k}(\widetilde{\mathcal{A}})\right)\left\|\mathbf{U}_{k\perp}^{*\top} \mathcal{M}_{k}(\widetilde{\mathcal{A}})\mathbb{P}_{\left(\mathbf{U}_{k}^{*\top} \mathcal{M}_{k}(\widetilde{\mathcal{A}})\right)^{\top}}\right\|_{\text{op}}}{\sigma_{r_{k}}^2\left(\mathbf{U}_{k}^{*\top}\mathcal{M}_{k}(\widetilde{\mathcal{A}})\right)-\sigma_{r_{k}+1}^2\left(\mathcal{M}_{k}(\widetilde{\mathcal{A}})\right)}\\&\lesssim\frac{\left((1-c) \sigma_{r_{1}}\left(\mathcal{M}_{k}(\mathcal{A}^{*})\right)+ a_{1}\sqrt{p_{2}p_{3} / n}\right) a_{1}\sqrt{p_{1} / n}}{\left((1-c) \sigma_{r_{1}}\left(\mathcal{M}_{k}(\mathcal{A}^{*})\right)+ a_{2}\sqrt{p_{2}p_{3} / n}\right)^2-a_{3}(p_{2}p_{3}+C\sqrt{p_{1}p_{2}p_{3}})}\\&\lesssim\frac{(1-c)
					\sigma_{r_{1}}\left(\mathcal{M}_{k}(\mathcal{A}^{*})\right)a_{1}\sqrt{p_{1} /n}+a_{1}^2\sqrt{p_{1}p_{2}p_{3}} / n}{(1-c)^2\sigma_{r_{1}}^2\left(\mathcal{M}_{k}(\mathcal{A}^{*})\right)-2\tau^2\sqrt{\frac{2t}{n}}p_{2}p_{3}/n-C\frac{\mathbb{E}[\eta_{i}^2]+\tau^2\sqrt{2t/n}}{n}\sqrt{p_{1}p_{2}p_{3}}}\\&\lesssim\frac{(1\!-\!c)\sigma_{r_{1}}\left(\mathcal{M}_{k}(\mathcal{A}^{*})\right)\left(\|\mathcal{A}^{*}\|_{F}\sqrt{p_{1}/n}\!+\!\sqrt[4]{\frac{2t}{n}}\|\mathcal{A}^{*}\|_{F}\sqrt{p_{1}/df}\right)\!+\!\|\mathcal{A}^{*}\|_{F}^2\frac{\bar{p}^{3/2}}{n}\!+\!\frac{\|\mathcal{A}^{*}\|_{F}^2}{df}\sqrt{\frac{2t\bar{p}^{3}}{n}}}
				{\sigma_{r_{1}}^2\left(\mathcal{M}_{k}(\mathcal{A}^{*})\right)},
	\end{aligned}
\end{equation}
where $a_{1}:=\sqrt{\|\mathcal{A}^{*}\|_{F}^2 + \tau^2\sqrt{\frac{2t}{n}}}$, $a_{2}:=\sqrt{\mathbb{E}[\psi_{\tau}(\eta_{i})^2] - \tau^2\sqrt{\frac{2t}{n}}}$ and $a_{3}:=\frac{\mathbb{E}[\psi_{\tau}(\eta_{i})^2]+\tau^2\sqrt{\frac{2t}{n}}}{n}$. Therefore, combining (\ref{eq6}) with (\ref{eq4}), by the union bound, we obtain that when $t=c\log(\bar{p})$, with probability at least $1-4\bar{p}^{-c}-6\exp(-c\bar{p})-2\exp\left(C_5 df\right)$ such that
\begin{align*}
	&\left\|\mathcal{A}^{(0)}-\mathcal{A}^{*}\right\|_{F} \\
	\leq{}& C_{4}\|\mathcal{A}^{*}\|_{F}\sqrt{\frac{df}{n}}+C_{4}\bar{\lambda}\sqrt{\bar{r}}\|\mathcal{A}^{*}\|_{F}\\&\times\frac{(1-c)\sigma_{r_{1}}\left(\mathcal{M}_{k}(\mathcal{A}^{*})\right)\left(\sqrt{\bar{p}/n}+\sqrt[4]{\frac{2\log(\bar{p})}{n}}\sqrt{\bar{p}/df}\right)+\|\mathcal{A}^{*}\|_{F}\frac{\bar{p}^{3/2}}{n}+\frac{\|\mathcal{A}^{*}\|_{F}}{df}\sqrt{\frac{2\log(\bar{p})\bar{p}^{3}}{n}}}{\sigma_{r_{1}}^2\left(\mathcal{M}_{k}(\mathcal{A}^{*})\right)}.
\end{align*}
Because of the inequalities $(x+y)^2\leq 2x^2+2y^2$ and $n\gtrsim \max\left\{ R\kappa^2\bar{p}^{\frac{3}{2}}\sqrt{\bar{r}},\kappa^6\bar{p}\bar{r}^2\right\}$, by Lemma E.2 of \cite{Han2022}, we have
\begin{align*}
	E^{(0)} &\leq 11 \kappa^2 b^{-6}\left\|\mathcal{A}^{(0)}-\mathcal{A}^{*}\right\|_{F}^2\\& \lesssim \kappa^2 b^{-6} \bar{\lambda}^2\bar{r}\left(\frac{\frac{\|\mathcal{A}^{*}\|_{F}^2p_{1}}{n}+\sqrt{\frac{2t}{n}}\frac{p_{1}\|\mathcal{A}^{*}\|_{F}^2}{df}}{\sigma_{r_{1}}^2\left(\mathcal{M}_{k}(\mathcal{A}^{*})\right)}+\frac{\frac{\|\mathcal{A}^{*}\|_{F}^4\bar{p}^3}{n^2}+\frac{2\|\mathcal{A}^{*}\|_{F}^4t\bar{p}^3}{df^2 n}}{\sigma_{r_{1}}^4\left(\mathcal{M}_{k}(\mathcal{A}^{*})\right)}\right)  \lesssim \frac{\underline{\lambda}^{1 / 2}}{\kappa^{3 / 2}}.
\end{align*}

Proof of step (2): We generalize Lemma 4 of \cite{Sun2020} to present that the restricted strong convexity condition is satisfied for the Huber loss under the set $\mathcal{C}(R):=\{\mathcal{A} \in \mathbb{R}^{p_{1} \times p_{2} \times p_{3}}:\|\mathcal{A}-\mathcal{A}^{*}\|_{F}\leq R,\text{rank}\left(\mathcal{A}-\mathcal{A}^{*}\right)\leq 2(r_{1},r_{2},r_{3})\}$ with high probability.

\begin{lemma}\label{lemma1}
	Suppose that all entries of $\mathcal{X}_{i}$ are i.i.d. sampled from sub-Gaussian distribution with mean-zero and variance-one. Then for all $\mathcal{A}\in \mathcal{C}(R)$, as long as the following conditions hold:
\begin{align*}
    \varpi \gtrsim \max \left \{ (4M)^{\frac{1}{2}}, 4c_{1}^2R \right \}, \quad n \gtrsim (\varpi/R)^2\left(\bar{p}^{3/2}\bar{r}+t\right),
\end{align*}
we have that

\begin{align*}
	\operatorname{P}\left(\left\langle\nabla\mathcal{L}_{\varpi}(\mathcal{A})-\nabla\mathcal{L}_{\varpi}(\mathcal{A}^{*}),\mathcal{A}-\mathcal{A}^{*}\right\rangle\geq\frac{4}{5}\left\|\mathcal{A}-\mathcal{A}^{*}\right\|_{F}^{2}\right)\geq 1-e^{-t},
\end{align*}
where $c_{1}\geq\sup_{\substack{\mathcal{V} \in \mathbb{R}^{p_{1} \times p_{2} \times p_{3}}\\ \|\mathcal{V}\|_{\mathrm{F}} \leq 1}}(\mathbb{E}\langle \mathcal{V},\mathcal{X}_{i}\rangle^4)^{1/4}$ is a positive constant.
\end{lemma}

\begin{proof}
	The difference with the proof of Lemma 4 in \cite{Sun2020} is that we denote
\begin{align*}
	\mathbb{Z}_{\mathcal{A}}:=\frac{\tau}{2 R n} \sum_{i=1}^{n} G_{i}^{\prime} \frac{\left\langle\mathcal{X}_{i}, \mathcal{A}-\mathcal{A}^{*}\right\rangle}
			{\left\|\mathcal{A}-\mathcal{A}^{*}\right\|_{F}},
\end{align*}
where $G_{i}^{\prime}\stackrel{i.i.d.}{\sim} N(0,1)$ and are independent with $\mathcal{X}_{i}$. Because $G_{i}^{\prime} \mathcal{X}_{i(j,k,l)}$ is a sub-exponential random variable with $\|G_{i}^{\prime} \mathcal{X}_{i(j,k,l)}\|_{\psi_{1}}=K$, by Theorem 2.5 of \cite{Boucheron2013}, we have that
\begin{align*}
	\mathbb{E}\left[\left|\frac{1}{n} \sum_{i=1}^{n} G_{i}^{\prime} \mathcal{X}_{i(j,k,l)}\right|^2\right]\lesssim\frac{K^2}{n} \quad\text{and}\quad \mathbb{E}\left[\left|\frac{1}{n} \sum_{i=1}^{n} G_{i}^{\prime} \mathcal{X}_{i(j,k,l)}\right|^4\right]\lesssim\frac{K^4}{n^2}.
\end{align*}
For $\{A_{i}\}_{i\in\left[\lceil \sqrt{p_3} \rceil\right]}\subset \mathbb{R}^{p_{1}\times p_{2}\left\lceil \sqrt{p_3} \right\rceil}$, denote 
\begin{align*}
\overline{\mathcal{M}}\left(\frac{1}{n} \sum_{i=1}^{n} G_{i}^{\prime} \mathcal{X}_{i}\right)=\left[\begin{array}{l:l:l:l}	A^{\top}_1 & A^{\top}_2 &\cdots&A^{\top}_{\left\lceil \sqrt{p_3} \right\rceil}\end{array}\right]^{\top}\in\mathbb{R}^{p_{1}\lceil \sqrt{p_3}\rceil\times p_{2}\lceil \sqrt{p_3} \rceil},	
\end{align*}
where
\begin{align*}
\left[\begin{array}{l:l:l:l}	A_1 & A_2 &\cdots&A_{\left\lceil \sqrt{p_3} \right\rceil}\end{array}\right]=\left[\begin{array}{l:l}\mathcal{M}_{1}\left(\frac{1}{n} \sum_{i=1}^{n} G_{i}^{\prime} \mathcal{X}_{i}\right)&\mathbf{0}_{p_{1}\times p_{2}\left(\lceil \sqrt{p_3} \rceil^2-p_{3}\right)}\end{array}\right].	
\end{align*}
Then, by \cite{Latala2005}, it follows that
\begin{align}
	\mathbb{E}\left[\left\|\overline{\mathcal{M}}\left(\frac{1}{n} \sum_{i=1}^{n} G_{i}^{\prime} \mathcal{X}_{i}\right)\right\|_{\text{op}}\right] & \lesssim_{K}\sqrt{\frac{p_{2}\lceil \sqrt{p_3} \rceil}{n}} + \sqrt{\frac{p_{1}\lceil \sqrt{p_3} \rceil}{n}}+\sqrt[4]{\frac{p_{1}\lceil \sqrt{p_3} \rceil^2 p_{3}}{n^2}} \notag \\
	&\lesssim_{K} \sqrt{\frac{\bar{p}^{3/2}}{n}}.\label{eq7}
\end{align}
Because of $\text{rank}\left(\overline{\mathcal{M}}\left(\mathcal{A}-\mathcal{A}^{*}\right)
			\right)\leq2r_{1}$, we get
\begin{align*}
	& \mathbb{E}\left\{\sup_{\mathcal{A}\in\mathcal{C}(R^{\prime})}\mathbb{Z}_{\mathcal{A}}\right\} \\
	={}&\frac{\varpi}{2 R^{\prime}}\mathbb{E}\left\{\sup_{\mathcal{A}\in\mathcal{C}(R^{\prime})} \frac{\left\langle\overline{\mathcal{M}}\left(\frac{1}{n} \sum_{i=1}^{n} G_{i}^{\prime} \mathcal{X}_{i}\right),\overline{\mathcal{M}}\left(\mathcal{A}-\mathcal{A}^{*}\right)\right\rangle}{\left\|\mathcal{A}-\mathcal{A}^{*}\right\|_{F}}\right\}\\
	\lesssim{}&\frac{\varpi}{ R^{\prime}}\mathbb{E}\left\{\sup_{\mathcal{A}\in\mathcal{C}(R^{\prime})} \left\|\overline{\mathcal{M}}\left(\frac{1}{n} \sum_{i=1}^{n} G_{i}^{\prime} \mathcal{X}_{i}\right)\right\|_{\text{op}}\left\|\overline{\mathcal{M}}\left(\mathcal{A}-\mathcal{A}^{*}\right)\right\|_{*}\!/\!\left\|\mathcal{A}-\mathcal{A}^{*}\right\|_{F}\right\}\\
	\lesssim{}& \frac{\varpi}{R^{\prime}}\mathbb{E}\left\{ \left\|\overline{\mathcal{M}}\left(\frac{1}{n} \sum_{i=1}^{n} G_{i}^{\prime} \mathcal{X}_{i}\right)\right\|_{\text{op}}\sqrt{r_{1}}\right\}\overset{(\ref{eq7})}{\lesssim_{K}} \frac{\varpi}{R^{\prime}}\sqrt{\frac{\bar{p}^{3/2}\bar{r}}{n}},
\end{align*}
where the first and second inequalities come from $\left \langle A,B\right\rangle\leq \|A\|_{\text{op}}\|B\|_{\star}$ where $A$ and $B$ are any two matrices of the same size, and $\|A\|_{\star}\leq \sqrt{\text{rank}(A)}\|A\|_{F}$.
\end{proof}

Based on Lemma \ref{lemma1} and $\varpi\asymp_{K,k}\left(Mn/df\right)^{\frac{1}{2}}$, we can get that at $T$-th iteration, $\mathcal{A}^{(T)}\in \mathcal{C}(\varpi)$ and (D.12) in the proof of Theorem 4 in \cite{Han2022} is satisfied with high probability. With the above two steps as a basis, by $\mathcal{A}^{(T)}\in \mathcal{C}(R)$, by Lemma E.6 of \cite{Han2022}, we obtain that as long as $n\gtrsim df$, for any rank-$(r_{1},r_{2},r_{3})$ tensor,
\begin{align*}
	& \left|\left\langle\mathcal{X}^{\prime}, \nabla \mathcal{L}_{\varpi}(\mathcal{A}^{(T)})-\nabla \mathcal{L}_{\varpi}\left(\mathcal{A}^{*}\right)\right\rangle\right| \\
	={}& \left|\left\langle\frac{1}{n}\mathcal{X}\left(\mathcal{X}^{\prime}\right),\psi_{\varpi}\left(\mathcal{X}(\mathcal{A}^{(T)})-y\right)-\psi_{\varpi}(\varepsilon)\right\rangle\right|\\
	\leq{}& \frac{1}{n}\left\|\mathcal{X}\left(\mathcal{X}^{\prime}\right)\right\|_{2}\left\|\psi_{\varpi}\left(\mathcal{X}(\mathcal{A}^{(T)})-y\right)-\psi_{\varpi}(\varepsilon)\right\|_{2}\\
	\leq{}& \frac{11}{9\sqrt{n}}\|\mathcal{X}^{\prime}\|_{F}\left\|\mathcal{X}(\mathcal{A}^{(T)}-\mathcal{A}^{\star})\right\|_{2}\leq\frac{121}{81}\|\mathcal{X}^{\prime}\|_{F}\left\|\mathcal{A}^{(T)}-\mathcal{A}^{\star}\right\|_{F} \\
	<{}& \frac{3}{2}\|\mathcal{X}^{\prime}\|_{F}\left\|\mathcal{A}^{(T)}-\mathcal{A}^{\star}\right\|_{F}
\end{align*}
with probability at least $1-C_{10}\exp(-C_{11}df)$. Finally, by following the lines of the proof of Theorem 4.2 in \cite{Han2022}, we obtain that
\begin{align*}
	\left\|\mathcal{A}^{(T+1)}-\mathcal{A}^{*}\right\|_{F}^2\leq C_{12}\left(\kappa^{4}\xi^2
		+\left(1-\frac{\eta_{0}}{1000\kappa^2}\right)^{T+1}\kappa^2\left\|\mathcal{A}^{(0)}
		-\mathcal{A}^{*}\right\|_{F}^2\right),
\end{align*}
where $\xi:=\sup_{\substack{\mathcal{T} \in \mathbb{R}^{p_{1} \times p_{2} \times p_{3}},\|\mathcal{T}\|_{\mathrm{F}} \leq 1,\\ \text{rank}(\mathcal{T})\leq (r_{1},r_{2},r_{3})}}\left|\left\langle\nabla \mathcal{L}_{\varpi}\left(\varepsilon_{i}\right), \mathcal{T}\right\rangle\right|$ and $\eta_{0}<c$, and $c$ is a small positive constant. Therefore, as long as $T_{\max}\gtrsim \log\left(\frac{\|\mathcal{A}^{(0)}-\mathcal{A}^{\star}\|_{F}}{\kappa\xi}\right)/\log\left(\frac{\kappa^2}{\kappa^2-2\rho\eta_{0}}\right)$, we have 
\begin{align*}
\left\|\mathcal{A}^{(T_{\max})}-\mathcal{A}^{*}\right\|_{F} \lesssim \kappa^{2}\xi.	
\end{align*}
By the proof of Theorem \ref{theorem2}, it follows that
\begin{align*}
	\operatorname{P}\left(\left\|\mathcal{A}^{(T_{\max})}-\mathcal{A}^{*}\right\|_{F}\lesssim \kappa^{2}\left(\frac{Mdf}{n}\right)^{\frac{1}{2}}
		\left(\sqrt{t}+t\right)\right)\geq 1-2\exp\left(df(\log(13)-t)\right).
\end{align*}

\section{Proof of Theorem \ref{theorem2}}
\label{sec:proof2}
Since
\begin{align*}
	\xi &:=  \sup_{\substack{\mathcal{T} \in \mathbb{R}^{p_{1} \times p_{2} \times p_{3}},\|\mathcal{T}\|_{\mathrm{F}} \leq 1,\\ \text{rank}(\mathcal{T})\leq (r_{1},r_{2},r_{3})}}\left|\left\langle\nabla \mathcal{L}_{\varpi}\left(\varepsilon_{i}\right), \mathcal{T}\right\rangle\right| \\
	&= \sup_{\substack{\mathcal{T} \in \mathbb{R}^{p_{1} \times p_{2} \times p_{3}},\|\mathcal{T}\|_{\mathrm{F}} \leq 1,\\ \text{rank}(\mathcal{T})\leq (r_{1},r_{2},r_{3})}}\left|\left\langle\frac{1}{n}\sum_{i=1}^{n}\ell'_{\varpi}(\varepsilon_{i})\mathcal{X}_{i}, \mathcal{T}\right\rangle\right|
\end{align*}
and  $\mathbb{E}\left[\varepsilon_{i}\mathcal{X}_{i(j,k,l)}\right]=\mathbb{E}\left[\mathcal{X}_{i(j,k,l)}\mathbb{E}\left[\varepsilon_{i}\big{|}\mathcal{X}_{i(j,k,l)}\right]\right]=0$ for $(j,k,l)\in [p_{1}]\times[p_{2}]\times[p_{3}]$, we have
\begin{align*}
	\left|\frac{1}{n}\sum_{i=1}^{n}\ell'_{\varpi}(\varepsilon_{i})\sum_{(j,k,l)}\mathcal{X}_{i(j,k,l)}\mathcal{T}_{(j,k,l)}\right|& \leq \left|\mathbb{E}\left[\ell'_{\varpi}(\varepsilon_{i})z_{i}\right]-\mathbb{E}\left[\varepsilon_{i}z_{i}\right]\right| \\
	&\quad +\left|\frac{1}{n}\sum_{i=1}^{n}\ell'_{\varpi}(\varepsilon_{i})z_{i}-\mathbb{E}\left[\ell'_{\varpi}(\varepsilon_{i})z_{i}\right]\right|.
\end{align*}
where $z_{i}:=\sum_{(j,k,l)}\mathcal{X}_{i(j,k,l)}\mathcal{T}_{(j,k,l)}$. For the first term, since $\left|\ell_{1}^{\prime}(x)-x\right| \lesssim x^2$, we have
\begin{align*}
	\left|\mathbb{E}\left[\ell'_{\varpi}(\varepsilon_{i})z_{i}\right]-\mathbb{E}\left[\varepsilon_{i}z_{i}\right]\right| &\leq\varpi\left| \mathbb{E}\left[\left\{\ell_{1}^{\prime}(\varepsilon_{i} / \varpi)-\varepsilon_{i} / \varpi\right\} z_{i}\right]\right| \\
	&\lesssim \varpi^{-1}\left|\mathbb{E}\left[ \varepsilon_{i}^2 z_{i}\right]\right| \lesssim \varpi^{-1} M K \sqrt{\frac{k}{k-1}}.
\end{align*}
For the second term, since $|\ell'_{\varpi}(\varepsilon_{i})|=\left|\varpi \ell_{1}^{\prime}\left(\varepsilon_i / \varpi\right)\right| \leq \min \left\{c_{1} \varpi, \left|\varepsilon_i\right|\right\}$, then for any $q\geq2$, 
\begin{align*}
	&\mathbb{E}\left|\ell'_{\varpi}(\varepsilon_{i})z_{i}\right|^q \\
	\leq{}& (c_{1}\varpi)^{q-2}\mathbb{E}|\varepsilon_{i}^{2}z_{i}^{q}| \leq (c_{1}\varpi)^{q-2}\sqrt[k]{\mathbb{E}\left(\mathbb{E}\left(|\varepsilon_{i}|^{2}|\mathcal{X}_{i}\right)^k\right)}\left(\mathbb{E}|z_{i}^{\frac{qk}{k-1}}|\right)^{\frac{k-1}{k}}\\
	\leq{}& (c_{1}\varpi)^{q-2}M\left(\frac{K^2qk}{k-1}\right)^{q/2}\leq q! M\left(eKc_{1}\varpi\sqrt{\frac{2 k}{k-1}}\right)^{q-2}.
\end{align*}
Therefore, by following the proof in \eqref{eq1}--\eqref{eq4}, we obtain that
\begin{align*}
\operatorname{P}\left(\xi\leq C_{5}\left(\frac{Mdf}{n}\right)^{\frac{1}{2}}
	\left(\sqrt{t}+t\right)\right)\geq 1-2\exp\left(df(\log(13)-t)\right).
\end{align*}

\begin{lemma}\label{lemma2}
	Consider $\{\eta_{i}\}_{i=1}^n$ are i.i.d. with $\mathbb{E}[\eta_{i}^{2}]\leq M$ and $\{X_{i}\}_{i=1}^{n}$ are $n$ i.i.d. random $d_{1}\times d_{2}$ matrices whose entries $X_{i(j,k)}$are i.i.d. sub-gaussian random variables with $\|X_{i(j,k)}\|_{\psi_{2}}=K$. Denote $A=\frac{1}{n}\sum_{i=1}^{n}\psi_{\tau}(\eta_{i})X_{i}$ where $\tau$ is given in Theorem \ref{theorem1}. Then there exist positive constants $C$ and $c$ depending only on $K$ such that
		\begin{align*}
			\operatorname{P}\left(\sigma_{\max}^2(A)\leq\frac{\mathbb{E}[\psi_{\tau}(\eta_{i})^2]+\tau^2\sqrt{\frac{2t}{n}}}{n}\left(\sqrt{d_{1}}+C\sqrt{d_{2}}+\sqrt{2t}\right)^2\right) & \geq 1-2e^{-ct}, \\
			\operatorname{P}\left(\sigma_{\min}^2(A)\geq\frac{\mathbb{E}[\psi_{\tau}(\eta_{i})^2]-\tau^2\sqrt{\frac{2t}{n}}}{n}\left(\sqrt{d_{1}}-C\sqrt{d_{2}}-\sqrt{2t}\right)^2\right) & \geq 1-2e^{-ct}.
		\end{align*}
\end{lemma}

\begin{proof}
This lemma is proved in a similar manner to the proof of Lemma 6 in \cite{Zhang2020}. Since $\|\frac{1}{n}\sum_{i=1}^{n}\psi_{\tau}(\eta_{i})X_{i(j,k)}\|_{\psi_{2}}\leq K\sqrt{\sum_{i=1}^n\psi_{\tau}(\eta_{i})^2}/n$ for given $\{\psi_{\tau}(\eta_{i})\}_{i=1}^n$, by Theorem 5.39 in \cite{Vershynin2012}, we obtain that
\begin{align*}
	\operatorname{P}\left(\sigma_{\max}^2(A)\geq\frac{\sum_{i=1}^n\psi_{\tau}(\eta_{i})^2}{n^2}\left(\sqrt{d_{1}}+C\sqrt{d_{2}}+\sqrt{2t}\right)^2\big{|}\{\psi_{\tau}(\eta_{i})\}_{i=1}^n\right)\leq e^{-ct}, \\
	\operatorname{P}\left(\sigma_{\min}^2(A)\leq\frac{\sum_{i=1}^n\psi_{\tau}(\eta_{i})^2}{n^2}\left(\sqrt{d_{1}}-C\sqrt{d_{2}}-\sqrt{2t}\right)^2\big{|}\{\psi_{\tau}(\eta_{i})\}_{i=1}^n\right)\leq e^{-ct}.
\end{align*}
For $\frac{1}{n}\sum_{i=1}^n\psi_{\tau}(\eta_{i})^2$, by Hoeffding's inequality, we can get that
\begin{align*}
	\operatorname{P}\left(\frac{1}{n}\sum_{i=1}^n\psi_{\tau}(\eta_{i})^2-\mathbb{E}[\psi_{\tau}(\eta_{i})^2]\geq \tau^2\sqrt{\frac{2t}{n}}\right) & \leq e^{-t}, \\
	\operatorname{P}\left(\frac{1}{n}\sum_{i=1}^n\psi_{\tau}(\eta_{i})^2-\mathbb{E}[\psi_{\tau}(\eta_{i})^2]\leq -\tau^2\sqrt{\frac{2t}{n}}\right) & \leq e^{-t}.
\end{align*}
Therefore,
\begin{align*}
	&\operatorname{P}\left(\sigma_{\max}^2(A)\leq\frac{\mathbb{E}[\psi_{\tau}(\eta_{i})^2]+\tau^2\sqrt{\frac{2t}{n}}}{n}\left(\sqrt{d_{1}}+C\sqrt{d_{2}}+\sqrt{2t}\right)^2\right) \\
	\geq{} & \operatorname{P}\left(\frac{1}{n}\sum_{i=1}^n\psi_{\tau}(\eta_{i})^2\leq \mathbb{E}[\psi_{\tau}(\eta_{i})^2]+\tau^2\sqrt{\frac{2t}{n}}\right)\\
	&\times\operatorname{P}\left(\sigma_{\max}^2(A)\leq\frac{\sum_{i=1}^n\psi_{\tau}(\eta_{i})^2}{n^2}\left(\sqrt{d_{1}}+C\sqrt{d_{2}}+\sqrt{2t}\right)^2\big{|}\{\psi_{\tau}(\eta_{i})\}_{i=1}^n\right) \\
	\geq{}& 1-2e^{-ct},
\end{align*}
\begin{align*}
	&\operatorname{P}\left(\sigma_{\min}^2(A)\geq\frac{\mathbb{E}[\psi_{\tau}(\eta_{i})^2]-\tau^2\sqrt{\frac{2t}{n}}}{n}\left(\sqrt{d_{1}}-C\sqrt{d_{2}}-\sqrt{2t}\right)^2\right) \\
	\geq{}& \operatorname{P}\left( \frac{1}{n}\sum_{i=1}^n\psi_{\tau}(\eta_{i})^2\geq \mathbb{E}[\psi_{\tau}(\eta_{i})^2]-\tau^2\sqrt{\frac{2t}{n}}\right)\\
	&\times \operatorname{P}\left(\sigma_{\min}^2(A)\geq\frac{\sum_{i=1}^n\psi_{\tau}(\eta_{i})^2}{n^2}\left(\sqrt{d_{1}}-C\sqrt{d_{2}}-\sqrt{2t}\right)^2\big{|}\{\psi_{\tau}(\eta_{i})\}_{i=1}^n\right) \\
	\geq{}& 1-2e^{-ct}.
\end{align*}
\end{proof}

\begin{lemma}\label{lemma3}
Consider $\{X_{i}\}_{i=1}^{n}$ are $n$ i.i.d. random $d_{1}\times d_{2}$ matrices whose entries are i.i.d. sub-gaussian random variables and $\{y_{i}\}_{i=1}^n$ are from Eq.~\eqref{eq}. Then there exists a constant $C$ such that with probability at least $1-2 e^{(d_{1}+d_{2})(\log(7)-t)}$,
\begin{align*}
	\left\|\frac{1}{n}\sum_{i=1}^{n}\psi_{\tau}(y_{i})X_{i}-\mathbb{E}\left[y_{i}X_{i}\right]\right\|_{\text{op}}\leq C\left(\frac{df}{Mn}
		\right)^{\frac{1}{2}}\left(\sqrt{(d_{1}+d_{2})t/df}+(d_{1}+d_{2})t/df+1\right).
\end{align*}
\end{lemma}

\begin{proof}
By the triangle inequality,
\begin{align*}
	& \left\|\frac{1}{n}\sum_{i=1}^{n}\psi_{\tau}(y_{i})X_{i}-\mathbb{E}\left[y_{i}X_{i}\right]\right\|_{\text{op}} \\
	\leq{}& \left\|\mathbb{E}\left[\psi_{\tau}(y_{i})X_{i}\right]-
		\mathbb{E}\left[y_{i}X_{i}\right]\right\|_{\text{op}}+\left\|\frac{1}{n}\sum_{i=1}^{n}\psi_{\tau}(y_{i})X_{i}-\mathbb{E}\left[\psi_{\tau}(y_{i})X_{i}\right]\right\|_{\text{op}}.
\end{align*}
Denote $z_{i}:=u^{\top}X_{i}v$ where $u \in \mathbb{S}^{d_{1}-1}$ and $v \in \mathbb{S}^{d_{2}-1}$. $\frac{1}{n}\sum_{i=1}^{n}\psi_{\tau}(y_{i})u^{\top}X_{i}v=\frac{1}{n}\sum_{i=1}^{n}\psi_{\tau}(y_{i})z_{i}$. Since $z_{i}$ is sub-Gaussian random variable with $\|z_{i}\|_{\psi_{2}}=K$, by using the similar treatment with (\ref{eq3}), we obtain that for some constants $C_{1}$ and $C_{2}$ depending only on $K, k$, and $c$,
\begin{equation}\label{eq5}
	\begin{aligned}
		& \left\|\mathbb{E}\left[\psi_{\tau}(y_{i})X_{i}\right] - \mathbb{E}\left[y_{i}X_{i}\right]\right\|_{\text{op}} \\
		={}& \sup_{u,v}|\mathbb{E}\left[\psi_{\tau}(y_{i})z_{i}\right]-\mathbb{E}\left[y_{i}z_{i}\right]|\leq C_{1}\left(M^{1/k}+c_{1}^{2}\|\mathcal{A}^{*}\|_{F}^{2}\right)/\tau \\
		\leq{}& C_{2}\left(M+c_{1}^{2}\|\mathcal{A}^{*}\|_{F}^{2}\right)\left(\frac{df}{Mn}\right)^{\frac{1}{2}}.
	\end{aligned}
\end{equation}
For the second term, according to (\ref{eq1}) and (\ref{eq2}), there exists an absolute constant $C_{3}$ depending on $K, k$ and $c$ such that
\begin{align*}
	\operatorname{P}\left(\left|\frac{1}{n}\sum_{i=1}^{n}\psi_{\tau}(y_{i})z_{i}-\mathbb{E}
		\left[\psi_{\tau}(y_{i})z_{i}\right]\right|\leq C_{3}\sqrt{\frac{Mt}{n}}+C_{3}\frac{\tau t}{n}\right) \geq 1-2e^{-t}.
\end{align*}
Let $\mathcal{N}^{d_{1}}_{\frac{1}{3}}$ and $\mathcal{N}^{d_{2}}_{\frac{1}{3}}$ be $\frac{1}{3}$-nets of $\mathbb{S}^{d_{1}-1}$ and $\mathbb{S}^{d_{2}-1}$ respectively, where $|\mathcal{N}^{d_{1}}_{1/3}|\leq 7^{d_{1}}$ and $|\mathcal{N}^{d_{2}}_{1/3}|\leq 7^{d_{2}}$. There exist $u_1\in \mathcal{N}^{d_{1}}_{\frac{1}{3}}$ and $v_1\in \mathcal{N}^{d_{2}}_{\frac{1}{3}}$ such that $\left\|u-u_1\right\|_2 \leq 1 / 3$ and $\left\| v-v_1 \right\|_2 \leq 1 / 3$. Thus, for any matrix $A \in \mathbb{R}^{d_1 \times d_2}$, we have
\begin{align*}
u^{\top} A v&=u_1^{\top}A v_{1}+\left(u-u_1\right)^{\top} A v_1+u_1^{\top} A\left(v-v_1\right)+\left(u-u_1\right)^{\top} A\left(v-v_1\right) \\ &\leq \sup _{u \in \mathcal{N}_{\frac{1}{3}}^{d_1}, v \in \mathcal{N}_{\frac{1}{3}}^{d_2}} u^{\top} A v+\left(\frac{1}{3}+\frac{1}{3}+\frac{1}{9}\right) \sup _{u \in \mathbb{S}^{d_1-1}, v \in \mathbb{S}^{d_2-1}} u^{\top} A v.
\end{align*}
Thus, $\|A\|_{\text {op }} \leq \frac{9}{2} \sup _{\substack{u \in \mathcal{N}_{\frac{1}{3}}^{d_1}, v \in \mathcal{N}_{\frac{1}{3}}^{d_2}}} u^{\top} A v$.  Therefore, the following inequality holds:
\begin{align*}
	\left\|\frac{1}{n}\sum_{i=1}^{n}\psi_{\tau}(y_{i})X_{i}-\mathbb{E}\left[\psi_{\tau}(y_{i})X_{i}\right]\right\|_{\text{op}}\leq\frac{9}{2}\max_{\substack{u\in\mathcal{N}^{d_{1}}_{\frac{1}{3}}\\v\in\mathcal{N}_{\frac{1}{3}}^{d_{2}}}}\left|\frac{1}{n}\sum_{i=1}^{n}\psi_{\tau}(y_{i})z_{i}-\mathbb{E}\psi_{\tau}(y_{i})z_{i}\right|.
\end{align*}
By using the union bound for all $u\in\mathcal{N}^{d_{1}}_{\frac{1}{3}}$ and $v\in \mathcal{N}^{d_{2}}_{\frac{1}{3}}$ and (\ref{eq6}), we have with probability at least $1-2 e^{(d_{1}+d_{2})(\log(7)-t)}$,
\begin{align}\label{eq9}
		& \left\|\frac{1}{n}\sum_{i=1}^{n}\psi_{\tau}(y_{i})X_{i} - \mathbb{E}\left[\psi_{\tau}(y_{i})X_{i}\right]\right\|_{\text{op}} \notag \\
		\leq{}& C_{4}\sqrt{\frac{M (d_{1}+d_{2})t}{n}}+C_{4}\frac{\tau (d_{1}+d_{2})t}{n} \notag \\
		\leq{}& C_{5}M^{\frac{1}{2}}\left(\frac{df}{n}\right)^{\frac{1}{2}}\left(\sqrt{(d_{1}+d_{2})t/df}+(d_{1}+d_{2})t/df\right).
\end{align}
Therefore, the conclusion holds by combining (\ref{eq5}) with (\ref{eq9}).
\end{proof}



\bibliographystyle{elsarticle-harv}
\begin{singlespace}
\bibliography{references}
\end{singlespace}





\end{document}